\documentclass[twoside,11pt]{article} 
\usepackage{jmlr2e} 

\usepackage{graphicx}

\usepackage{setspace,dsfont,enumerate,color,bm,float,epstopdf,verbatim,subfiles,lipsum,overpic}    
\usepackage{amsmath}
\usepackage{amssymb}
\usepackage[outercaption]{sidecap}
\usepackage{tabularx}
\usepackage[table]{xcolor}
\usepackage{pgffor}
\usepackage{comment,tikz-cd}
\usepackage[T1]{fontenc}
\usepackage[utf8]{inputenc}
\usepackage[normalem]{ulem}

\newtheorem{assumptions}{Assumptions}{\bf}{\rm}
\newtheorem{algorithm}[theorem]{Algorithm}

\newcommand{\bbP}{\mathbb P}
\newcommand{\bbR}{\mathbb R}
\newcommand{\bbC}{\mathbb C}

\renewcommand{\Delta}{\triangle}
\newcommand{\eps}{\varepsilon} 
\newcommand{\dee}{\mathrm{d}}

\newcommand{\tr}{\mathrm{Tr}}

\newcommand{\EE}{\mathbb{E}}

\newcommand{\PP}{\mathbb{P}}
\newcommand{\R}{\mathbb{R}}
\newcommand{\N}{\mathbb{N}}

\newcommand{\cH}{\mathcal{H}}
\newcommand{\sP}{\mathsf{P}}

\newcommand{\bR}{\mathbf{R}}
\newcommand{\bC}{\mathbf{C}}
\newcommand{\bu}{\mathbf{u}}

\jmlrheading{1}{2018}{1-45}{4/00}{10/00}{dunlop18}{M.M.Dunlop and M.A.Girolami and A.M.Stuart and A.Teckentrup}
\ShortHeadings{Deep Gaussian Processes}{Dunlop, Girolami, Stuart, and Teckentrup}
\firstpageno{1}

\begin{document}

\title{How Deep Are Deep Gaussian Processes?}
\author{\name M.M.Dunlop \email mdunlop@caltech.edu\\ \addr Computing and Mathematical Sciences\\ Caltech\\ Pasadena\\ CA 91125, USA
\AND
\name M.A. Girolami \email m.girolami@imperial.ac.uk\\ \addr Department of Mathematics\\ Imperial College London\\SW7 2AZ, UK\\and\\The Alan Turing Institute\\96 Euston Road\\London, NW1 2DB, UK
\AND
\name A.M. Stuart \email astuart@caltech.edu\\ \addr Computing and Mathematical Sciences\\ Caltech\\ Pasadena\\ CA 91125, USA
\AND
\name A.L. Teckentrup \email a.teckentrup@ed.ac.uk\\ \addr School of Mathematics\\ University of Edinburgh\\ Edinburgh, EH9 3FD, UK\\and\\The Alan Turing Institute\\96 Euston Road\\London, NW1 2DB, UK
}

\editor{Kevin Murphy and Bernhard Sch{\"o}lkopf}

\maketitle

\begin{abstract} Recent research has shown the potential
utility {of Deep Gaussian Processes. These deep structures are} probability distributions, designed through hierarchical construction, which are conditionally Gaussian. In this paper, the current published body of work is placed  in a common framework and, through recursion, several classes of deep Gaussian processes are defined. The {resulting} samples {generated from a deep Gaussian process} have a Markovian structure 
with respect to the depth parameter, and the effective depth of the resulting process is 
interpreted in terms of the ergodicity, or non-ergodicity, of the resulting Markov chain. 
{For the classes of deep Gaussian processes introduced, we provide results concerning their ergodicity and hence their effective depth.
We also demonstrate how these processes may be used for inference; in particular
we show how a Metropolis-within-Gibbs construction across the levels of the
hierarchy can be used to derive sampling tools which are robust to the
level of resolution used to represent the functions on a computer. For illustration, we consider the effect of ergodicity in some simple numerical examples.}
\end{abstract}

%\tableofcontents

\section{Introduction}
\label{sec:I}
\subsection{Background}
\label{ssec:B}
Gaussian processes have proved remarkably successful as a tool for various statistical inference
and machine learning tasks \citet{williams2006gaussian,ko01,hkccr04,stein}. This success relates
in part to the ease with which computations may be performed in the Gaussian
framework, and also to the flexible ways in which Gaussian processes 
may be used, for example when combined with thresholding to perform classification tasks
via probit models \citet{neal1997monte,williams2006gaussian} 
or to find interfaces in Bayesian inversion
\citet{ILS16}. {Nonetheless there are limits to the sort of phenomena that are
readily expressible via direct use of Gaussian processes, such as in the sparse data scenario, where the constructed probability distribution is far from posterior contraction}. Recognizing this 
fact, there have been a number of interesting research activities 
which seek to represent
new phenomena via the hierarchical cascading of Gaussians. Early work of this type 
includes the PhD thesis \citet{paciorek_thesis} (see also \citet{paciorek2004nonstationary}) in which the aim is to reproduce spatially
non-stationary phenomena, and this is achieved by means of a Gaussian process
whose covariance function itself depends on another Gaussian process. This
idea was recently re-visited in \citet{lassi}, using the precision operator
viewpoint, rather than covariance function, and building on the explicit link
between Gaussian processes and stochastic partial differential equations
(SPDEs) \citet{matern_spde}. A different approach was adopted in \citet{damianou2013deep}
where a Gaussian process was directly composed with another Gaussian process;
furthermore the idea was implemented recursively, leading to what is referred to as deep Gaussian processes (DGP).
These ingenious constructions open up new possibilities for problems in non-parametric
inference and machine learning and the purpose of this paper is to establish, and
utilize, a common framework for their study. Relevant to our analysis is the early work in \citet{diaconis_freedman_1999} 
which studied iterations of random Lipschitz functions and the conditions required for their convergence.

\subsection{Our Contribution}
\label{ssec:O}
In the paper we make three main contributions:

\begin{itemize}
\item We demonstrate a unifying perspective on the hierarchical Gaussian processes
described in the previous subsection, leading to a wide class of deep Gaussian
processes, with a common framework within which new deep Gaussian processes can
be constructed.
\item By exploiting the fact that this common framework has a Markovian structure,
we interpret the depth of the process in terms of the ergodicity or
non-ergodicity of this process; in simple terms ergodic constructions have effective
depth given by the mixing time. 
\item We demonstrate how these processes may be used for inference; in particular
we show how a Metropolis-within-Gibbs construction across the levels of the
hierarchy can be used to derive sampling tools which are robust to the
level of resolution used to represent the functions on a computer.
\end{itemize}

We also describe numerical experiments which illustrate the theory,
and which demonstrate some of the limitations of the framework in 
the inference context, suggesting the need for further algorithmic
innovation and theoretical understanding.
{We now summarize the results and contributions by direct reference
to the main theorems in the paper.}{
\begin{itemize}
\item Theorem \ref{thm:triv} shows that a composition-based deep Gaussian process 
will, with sufficiently many layers, produce samples that are approximately constant. This pathology can be avoided by, for example, increasing the width of each hidden layer, or allowing each layer to depend on the input layer.

\item Theorem \ref{t:ergodic1} shows the ergodicity of a class of discretized deep Gaussian processes, constructed using non-stationary covariance \emph{functions}. As a consequence, there is little benefit in adding additional layers after a certain point. This observation elucidates the mechanism underlying the choices of DGPs with a small number of layers for inference in numerous papers, for example in \citet{cutajar2016random,salimbeni2017doubly,dai2015variational}.

\item Theorem \ref{t:added} establishes a similar result as Theorem \ref{t:ergodic1} on function space, for a different class of deep Gaussian processes constructed using non-stationary covariance \emph{operators}.

\item Theorem \ref{T:convolution} establishes the asymptotic properties of a deep Gaussian process formed by iterated convolution of fairly general classes of Gaussian random fields. Specifically it is shown that such processes will either converge weakly to zero or diverge as the number of layers is increased, and so they will provide little flexibility for inference in practice.
\end{itemize}}

\subsection{Overview}
\label{ssec:O2}
The general framework in which we place the existing literature, and which we
employ to analyze deep Gaussian processes, and to construct algorithms 
for related inference tasks, is as follows. We consider sequences of functions
$\{u_n\}$ which are conditionally Gaussian:
\begin{equation}
\label{eq:cg1}
\tag{CovOp}
u_{n+1}|u_n \sim N\bigl(m(u_n),C(u_n)\bigr);
\end{equation}
here $m(u_n)$ denotes the mean function and $C(u_n)$ the covariance operator.
We will also sometimes work with the covariance function representation,
in which case we will write
\begin{equation}
\label{eq:cg2}
\tag{GP}
u_{n+1}|u_n \sim \text{GP}\bigl(m(x;u_n),c(x,x';u_n)\bigr).
\end{equation}
{Note that the covariance function is the kernel of the covariance operator
when the latter is represented as an integral operator over the approximate domain
$D \subseteq \bbR^d$:
$$\bigl(C(u_n)\phi\bigr)(x)=\int c(x,x';u_n)\phi(x')dx'.$$
In most of the paper we consider the centred case where $m \equiv 0$,
although the flexibility of allowing for non-zero mean will be important
in some applications, as discussed in the conclusions.}  
When the mean is zero, the iterations 
\eqref{eq:cg1} and \eqref{eq:cg2} can be written in the form
\begin{equation}
\label{eq:cg3}
\tag{ZeroMean}
u_{n+1}=L(u_n)\xi_{n+1},
\end{equation}
{where $\{\xi_n\}$ form an i.i.d. Gaussian sequence and, for each $u$, $L(u)$ is
a linear operator. For example if the $\xi_n$ are white then the covariance
operator is $C(u)=L(u)L(u)^\top$ with $\top$ denoting the adjoint operation and
$L(u)$ is a Cholesky factor of $C(u).$
The formulation \eqref{eq:cg3} is useful in much of our analysis.
For the purpose of this paper, we will refer to any sequence of functions constructed as in \eqref{eq:cg3} as a deep Gaussian process.
}

In section \ref{sec:T} we discuss the hierarchical Gaussian constructions
referenced above, and place them in the setting of equations
\eqref{eq:cg1}, \eqref{eq:cg2} and \eqref{eq:cg3}.
Section \ref{sec:E} studies the ergodicity of the resulting deep Gaussian
processes, using the Markov chain which defines them. 
In section \ref{sec:I2} we provide supporting numerical experiments;
we give illustrations of draws from deep Gaussian process priors,
and we discuss inference. In the context of inference  we describe
a methodology for
MCMC, using deep Gaussian priors, which is defined in the function space limit
and is hence independent of the level of
resolution used to represent the functions $u_n;$ numerical
illustrations are given.
We conclude in section \ref{sec:C} in which we describe generalizations
of the settings considered in this paper, and highlight future directions.

{
\subsection{Notation}
\label{ssec:notation}

The structure of the deep Gaussian processes above means that they can be interpreted as Markov chains on a Hilbert space $\cH$ of functions. Let $\mathcal{B}(\cH)$ denote the Borel $\sigma$-algebra on $\cH$.
We denote by $\sP:\cH\times\mathcal{B}(\cH)\to\R$ the one-step transition probability distribution,
\begin{equation}
\label{eq:k1}
\sP(u,A) = \mathbb{P}(u_{n} \in A\,|\,u_{n-1} = u),
\end{equation}
and denote by $\sP^n:\cH\times\mathcal{B}(\cH)\to\R$ the $n$-step transition probability distribution,
\begin{equation}
\label{eq:k2}
\sP^n(u,A) = \mathbb{P}(u_n \in A\,|\,u_0 = u).
\end{equation}
Thus, for example, in the case of the covariance operator construction
\eqref{eq:cg1} we have
$$\sP(u,\cdot) = N\bigl(0,C(u)\bigr),$$
when the mean is zero.
This Markovian structure will be exploited when showing ergodicity, 
or lack of ergodicity, of the chains.  
}

\section{Four Constructions}
\label{sec:T}

This section provides examples of four constructions of deep Gaussian processes,
all of which fall into our general framework. The reader will readily design others.

\subsection{Composition}
\label{ssec:C}
{Let $D \subseteq \R^d$, $D'\subseteq \R^l$, $u_n: D \to \R^m$ and 
$F: \R^m \to D'$. If $\{\xi_n\}$ is a collection of i.i.d.
centred Gaussian processes taking values in the space of continuous functions
$C(D';\R^m)$ then we define the Markov chain }
\begin{equation}
\label{eq:MC1}
u_{n+1}(x) =\xi_{n+1}\Bigl(F\bigl(u_n(x)\bigr)\Bigr).
\end{equation}
{The case $m=l$, $F={\rm id}$ and $D=D'=\R^m$} was introduced in \citet{damianou2013deep}
and the generalization here 
is inspired by the formulation in \citet{duvenaud2014avoiding}. d{The case where two layers are employed could be interpreted as a form of warped Gaussian process: a generalization of Gaussian processes that have been used successfully in a number of inference problems \citet{snelson2004warped,schmidt2003bayesian}}

We note that the mapping $\xi \mapsto \xi \circ F \circ u$ is linear, and
we may thus define $L(u)$ by $L(u)\xi=\xi \circ F \circ u$; hence
the Markov chain may be written in the form \eqref{eq:cg3}. If $\xi_1 \sim N(0,\Sigma)$
then the Markov chain has the form \eqref{eq:cg1}, with mean zero
and  $C(u)=L(u)\Sigma L(u)^*;$
if {$\xi_1 \sim \text{GP}\bigl(0,k(z,z')\bigr)$} 
then the Markov chain has the form \eqref{eq:cg2} with
mean zero and $c(x,x';u)=k\Bigl(F\bigl(u(x)\bigr),F\bigl(u(x')\bigr)\Bigr).$

\subsection{Covariance Function}
\label{ssec:F}
Paciorek \citet{paciorek_thesis} gives a general strategy to construct 
anisotropic versions of isotropic covariance functions.
Let $\Sigma : \R^d \rightarrow \R^{d \times d}$ be such that $\Sigma(z)$ is symmetric positive definite for all $z \in \R^d$, and define the quadratic form
\[
Q(x,x') = (x-x')^T\left(\frac{\Sigma(x) + \Sigma(x')}{2}\right)^{-1} (x-x'), \qquad x,x' \in \R^d.
\]
If the isotropic correlation function $\rho_{S}(\cdot)$ is positive definite on $\R^d$, for all $d \in \N$, then the function 
\[
c(x,x') = \sigma^2\frac{2^\frac{d}{2} \det(\Sigma(x))^\frac{1}{4} \det(\Sigma(x'))^\frac{1}{4}}{\det(\Sigma(x) + \Sigma(x'))^\frac{1}{2}} \rho_{S}(\sqrt{Q(x,x')})
\]
is positive definite on $\R^d \times \R^d$ and may thus be used as a covariance function.
We make these statements precise below.
If we choose $\Sigma$ to depend on $u_n$ then this may be used as the
basis of a deep Gaussian process. To be concrete we choose
$$\Sigma(x)=F\bigl(u(x)\bigr)I_{d}$$
where $F:\R \to \R_{\geq 0}$ for $u:D \subseteq \R^d \to \R.$ We then write $c(x,x';u)$.
Now let $u_n:D \to \R$ and consider the Markov chain \eqref{eq:cg2}
in the mean zero case.
In \citet{paciorek_thesis} this iteration was considered over one-step
with $u_0 \sim \text{GP}(0,\sigma^2\rho_{S}(\|x-x'\|))$ and $u_1$ was
shown to exhibit interesting non-stationary effects. Here we generalize and
consider the deep process that results from this construction 
for arbitrary $n \in \N.$ By considering the covariance operator
$$\bigl(C(u)\varphi\bigr)(x)=\int_{\R^d} c(x,x';u)\varphi(x')\,\dee x'$$
we may write the iteration in the form \eqref{eq:cg1}. 
The form \eqref{eq:cg3} follows with $L(u)=C(u)^{\frac12}$ and $\xi_{n+1}$ being
white noise.
 
Various generalizations of this construction are possible, for example allowing
the pointwise variance of the process $\sigma^2$ to be spatially varying {\citet{heinonen2016non}}
and to depend on $u_n(x).$ These may be useful in applications, but we
confine our analysis to the simpler setting for expository purposes; however in Remark \ref{rem:cov} we discuss this generalization.

In order to make the statements made above precise,
let $\rho_\mathrm{S} : [0,\infty) \rightarrow \R$ be a stationary covariance kernel, where the covariance between locations $x$ and $y$ depends only on the Euclidean distance $\|x-y\|_2$. We make the following assumption on $\rho_\mathrm{S}$.

\begin{assumptions}\label{ass:posdef_stat}
\begin{enumerate}
\item[(i)] The covariance kernel $\rho_\mathrm{S}(\|x-y\|_2)$ is positive definite\footnote{If the double sum in this definition is only non-negative, we say that the kernel $\rho_S$ is positive semi-definite.
We are thus adopting the terminology used by Wendland \citet{wendland}, where the kernel $\rho_\mathrm{S}$ is called positive definite if the double sum in Assumptions \ref{ass:posdef_stat}(i) is positive, and positive semi-definite if the sum is non-negative. For historical reasons, there is an alternative terminology, used in for example \citet{paciorek_thesis}, where our notion of positive definite is referred to as strictly positive definite, and our notion of positive semi-definite is referred to as positive definite.}
on $\R^d \times \R^d$: for any $N \in \N$, $b \in \R^N\backslash\{0\}$ and pairwise distinct $\{x_i\}_{i=1}^N \subseteq \R^d$, we have
\[
\sum_{i=1}^N \sum_{j=1}^N b_i b_j \rho_\mathrm{S}\big(\|x_i-x_j\|_2\big) > 0.
\]
\item[(ii)] $\rho_\mathrm{S}$ is normalized to be a correlation kernel, i.e. $\rho_\mathrm{S}(0) = 1$.
\end{enumerate} 
\end{assumptions}

Using \citet[Theorem 6.11]{wendland}, sufficient conditions for $\rho_\mathrm{S}$ to fulfill Assumptions \ref{ass:posdef_stat}(i) are that $\rho_\mathrm{S}$, as a function of $x-y$, is continuous, bounded and in $L_1(\R^d)$, with a Fourier transform that is non-negative and non-vanishing. These sufficient conditions are satisfied, for example, for the family of Mat\`ern covariance functions and the Gaussian covariance. To satisfy Assumptions \ref{ass:posdef_stat}(ii), any positive definite kernel $\tilde \rho_\mathrm{S}$ can simply be rescaled by $\tilde \rho_\mathrm{S}(0)$.

We now have the following proposition,
a slightly weaker version of which is proved in \citet{paciorek_thesis}, where it is shown that $\rho(\cdot, \cdot)$ is positive semi-definite if $\rho_\mathrm{S}$ is positive semi-definite. Our proof, which is in the Appendix,
follows closely that of \citet[Theorem 1]{paciorek_thesis}, 
but sharpens the result using a characterization of 
positive definite kernels proved in \citet[Theorem 7.14]{wendland}.

\begin{proposition}\label{thm:nonstat} Let Assumptions \ref{ass:posdef_stat} hold. Suppose $\Sigma : \R^d \rightarrow \R^{d \times d}$ is such that $\Sigma(z)$ is symmetric positive definite for all $z \in \R^d$, and define the quadratic form
\[
Q(x,x') = (x-x')^T\left(\frac{\Sigma(x) + \Sigma(x')}{2}\right)^{-1} (x-x'), \qquad x,x' \in \R^d.
\]
Then the function $\rho(\cdot, \cdot)$, defined by
\[
\rho(x,x') = \frac{2^\frac{d}{2} |\Sigma(x)|^\frac{1}{4} |\Sigma(x')|^\frac{1}{4}}{|\Sigma(x) + \Sigma(x')|^\frac{1}{2}} \rho_\mathrm{S}\Big(\sqrt{Q(x,x')}\Big), 
\]
is positive definite on $\R^d \times \R^d$, for any $d \in \N$, and is a non-stationary correlation function.
\end{proposition}

Non-stationary covariance functions $c(x,y)$, for which $c(x,x) \neq 1$, can be obtained from the non-stationary correlation function $\rho(x,y)$ through multiplication by a standard deviation function $\sigma : \R^d \rightarrow \R$, in which case we have $c(x,y) = \sigma(x) \sigma(y) \rho(x,y)$. Since the product of two positive definite kernels is also positive definite by \citet[Theorem 6.2]{wendland}, the kernel $c(x,y)$ can be ensured to be positive definite by a proper choice of $\sigma$. We discuss generalizations such as this in the
conclusions section \ref{sec:C}.

We are interested in studying the behaviour of Gaussian processes with non-stationary correlation functions $\rho(x,y)$ of the form derived in Proposition \ref{thm:nonstat}, in the particular case where the matrices $\Sigma(z)$ are derived from another Gaussian process. 
Specifically, we consider the following hierarchy of conditionally Gaussian processes 
on a bounded domain $D \subseteq \R^d$ defined as follows:
\begin{subequations}
\label{eq:pac}
\begin{align}
u_0 &\sim \mathrm{GP}(0, \rho_{\mathrm{S}}(\cdot)), \\
u_{n+1} |u_{n} &\sim \mathrm{GP}(0, \rho(\cdot, \cdot;u_{n})), \quad \text{for } n \in \N.
\end{align} 
\end{subequations}
Here, $\rho(\cdot, \cdot;u_{n})$ denotes a non-stationary correlation function constructed from $\rho_\mathrm{S}(\cdot)$ as in Proposition \ref{thm:nonstat}, with the map $\Sigma$ defined through $u_{n}$. Typical choices for $\Sigma$ are $\Sigma(z) = \left(u_{n}(z)\right)^2 \, \mathrm{I}_d$ and $\Sigma(z) = \exp(u_{n}(z)) \, \mathrm{I}_d$. Choices such as the
first of these lead to the possibility of positive semi-definite $\Sigma$ and, in the
worst case, $\Sigma \equiv 0.$ If $\Sigma \equiv 0$ the resulting correlation function is
given by
$$\rho_{\mathrm{S}}(0)=1, \quad \text{and} \quad \rho_{\mathrm{S}}(r)=0 \quad \text{for any } r>0.$$
This does not correspond to any (function valued) Gaussian process on $\R^d$ \citet{kallianpur}: heuristically the
resulting process would be a white noise 
process, but normalized to  zero. However, it is possible to sample from 
any set of 
finite dimensional distributions when $\Sigma \equiv 0$: 
the correlation matrix is then the identity.
To allow for the possibility of $F(\cdot)$ taking the value zero,
we therefore only study the finite dimensional process defined as follows:
\begin{subequations}
\label{eq:pac2}
\begin{align}
\bu_0 &\sim {N}(0, \bR_\mathrm{S}), \\
\bu_{n+1} \; | \; \bu_{n} &\sim {N}(0, \bR(\bu_{n})), \quad \text{for } n \in \N.
\end{align} 
\end{subequations}
The vector $\bu_{n}$ has entries $(\bu_n)_i = u_n(x_i)$. Here, $\bR_\mathrm{S}$ is the covariance matrix with entries $(\bR_\mathrm{S})_{ij} = \rho_\mathrm{S}(\|x_i-x_j\|_2)$, and $\bR(\bu_{n})$ is the covariance matrix with entries $(\bR(\bu_{n}))_{ij} = \rho(x_i, x_j;u_{n})$. 
The set $\{x_j\}$ comprises a finite set of points in $\R^d.$

We may now generalize Proposition \ref{thm:nonstat} to allow for $\Sigma$ becoming zero.
In order to do this we make the following assumptions:

\begin{assumptions} \label{ass:posdef}
\begin{enumerate}[{(i)}]
\item We have $\Sigma(z) = G(z) \mathrm{I}_d$, for some non-negative, bounded function $G : \R \rightarrow \R_{\geq 0}$. %applied point-wise to $u_{n}$.
\item The correlation function $\rho_\mathrm{S}$ is continuous, with $\lim_{r \rightarrow \infty} \rho_{\mathrm{S}}(r) = 0$.
\end{enumerate}   
\end{assumptions}

We then have the following result on the positive-definiteness of $\rho(\cdot,\cdot)$,

\begin{proposition}\label{lem:posdef} Let Assumptions \ref{ass:posdef_stat} and \ref{ass:posdef} hold. Then the kernel $\rho(\cdot,\cdot)$ defined in  Proposition \ref{thm:nonstat}
is positive definite on $\R^{d} \times \R^{d}.$ 
\end{proposition}

\begin{remark}
\label{rem:new}
This proposition applies to the process \eqref{eq:pac2} with $\Sigma(z) = F\bigl(u_n(z)\bigr) \mathrm{I}_d$ and $F : \R \rightarrow \R_{\geq 0}$ locally bounded, by taking $G=F\circ u_n$, proving that $\rho(\cdot,\cdot;u_{n})$  
is positive definite on $D \times D$ 
for all bounded functions $u_{n}$ on $D$.
Here we generalize the notion of positive-definite in the obvious way
to apply on $D \subseteq \R^d$ rather than on the whole of $\R^d.$
\end{remark}

\subsection{Covariance Operator}
\label{ssec:O3}
Here we demonstrate how precision (inverse covariance) operators may be used to make deep
Gaussian processes. Because precision operators encode conditional independence
and sparsity this can be a very attractive basis for fast computations \citet{matern_spde}.
Our approach is inspired by the hierarchical
Gaussian process introduced in \citet{lassi}, where one-step of the
Markov chain which we introduce here was considered.
Let $D \subseteq \R^d$, $u_n: D \to \R$ and 
$X:=C(D;\R)$. Assume that
$F:\R \to \R_{\ge 0}$ is a bounded function. 
Let $C_{-}$ be a covariance operator associated to a Gaussian process taking
values in $X$ and let $P$ be the associated precision operator. 
Define the multiplication operator $\Gamma(u)$ by $\bigl(\Gamma(u)v\bigr)(x)=F\bigl(u(x)\bigr)v(x)$ and the covariance operator $C(u)$ by
$$C(u)^{-1}=P+\Gamma(u)$$
and consider the Markov chain \eqref{eq:cg1} with mean zero; this defines our deep Gaussian process.
We note that formulation \eqref{eq:cg2} can be obtained by observing that the
covariance function $c(u):=c(x,x';u)$ is the Green's function associated with the
precision operator for $C(u)$:
$$C(u)^{-1}c(\cdot,x';u)=\delta_{x'}(\cdot)$$
where $\delta_{x'}$ is a Dirac delta function centred at point $x'$.
Computationally we will typically choose $P$ to be a differential operator,
noting that then fast methods may be employed to sample the Gaussian process
$u_{n+1}|u_n$ by means of SPDEs \citet{matern_spde,stuart2010inverse}. d{If $P$ is chosen as a differential operator, then the order of this operator will be related to the order of regularity of samples, and $F$ will be related to the length scale of the samples. These relations are made explicit in the case of certain Whittle-Mat\'ern distributions when $F$ is constant \citet{matern_spde}; some boundary effects may be present when $D \neq \R^d$, though methodology is available to ameliorate these \citet{daon2016mitigating}.}
As in the previous subsection, the form \eqref{eq:cg3} 
follows with $L(u)=C(u)^{\frac12}$ and $\xi_{n+1}$ being white noise.
%In later sections it will be useful to consider the covariance operator $C_+$ 
%defined by 
%$$C_+^{-1}=P+F_+I,$$
%where $F_+$ is the supremum of $F$.

Generalizations of the construction in this subsection are possible,
and we highlight
these in subsection \ref{sec:C}; however for expository purposes we confine
our analysis to the setting described in this subsection.
For theoretical investigation of the equivalence, as measures, of
Gaussians defined by addition of an operator to 
a given precision operator, see \citet{pinski2015kullback}.

\subsection{Convolution}
\label{ssec:conv}

We consider the case (\ref{eq:cg3}) where $L(u)\xi := u*\xi$ is a 
convolution. To be concrete we let $D=[0,1]^d$ and construct
a sequence of functions $u_n:D \to \R$ (or $u_n:D \to \bbC$) defined via the iteration 
\[
u_{n+1}(x) = (u_n * \xi_{n+1})(x) := \int_{[0,1]^d} u_n(x-y)\xi_{n+1}(y)\,dy,
\]
where $\{\xi_n\}$ are a sequence of i.i.d. centred real-valued
Gaussian random functions on $D$.
Here we implicitly work with periodic extension of $u_n$ from $D$ to 
the whole of $\R^d$ in order to define the convolution.

\section{The Role of Ergodicity}
\label{sec:E}

The purpose of this section is to demonstrate that the iteration \eqref{eq:cg3}
is, in many situations, ergodic. This has the practical implication that the
effective depth of the deep Gaussian process is limited by the mixing time of
the Markov chain. In some cases the ergodic behaviour may be trivial (convergence
to a constant). Furthermore, even if the chain is not ergodic, the large iteration number dynamics may
blow-up, prohibiting use of the iteration at significant depth. The take home message
is that in many cases the effective depth is not that great. Great care will be needed
to design deep Gaussian processes whose depth, and hence approximation power, is
substantial. This issue was first identified in \citet{duvenaud2014avoiding},
and we here provide a more general analysis of the phenomenon within the
broad framework we have introduced for deep Gaussian processes.

\subsection{Composition}
\label{ssec:C2}

We first consider the case where the iteration is defined by (\ref{eq:MC1}), which includes examples considered in \citet{damianou2013deep, duvenaud2014avoiding}. In \citet{duvenaud2014avoiding} it was observed that after a number of iterations, sample paths are approximately piecewise constant. We investigate this effect in the context of ergodicity. We first make two observations:
\begin{enumerate}[(i)]
\item if $u_0$ is piecewise constant, then $u_n$ is piecewise constant for all $n \in \N$;
\item if $u_0$ has discontinuity set $\mathsf{Z}_0$, and $\mathsf{Z}_n$ denotes the discontinuity set of the $n$th iterate, then $\mathsf{Z}_{n+1}\subseteq \mathsf{Z}_n$ for all $n \in \N$.
\end{enumerate}
Due to point (ii) above, if the sequence $\{u_n\}$ is to be ergodic, then necessarily it must be the case that $\mathsf{Z}_n\to\varnothing$, or else the process will have retained knowledge of the initial condition. In particular, if the initial condition is piecewise constant, then ergodicity would force the limit to be constant in space.

In what follows we assume that the iteration is given by
{\[
u_{n+1}(x) = \xi_{n+1}\big(u_n(x)\big),\quad \xi_{n+1}^j \sim \mathrm{GP}\big(0,h(\|x-x'\|_2)\big)\;\text{i.i.d}.
\]}
where $h$ is a stationary covariance function. We therefore make the choice 
$m = l$ and $F = \mathrm{id}$ in (\ref{eq:MC1}) so that we are in the same setup as \citet{damianou2013deep, duvenaud2014avoiding}; the inclusion of more general maps $F$ is discussed in Remark \ref{rem:comp_F}. Then for any $x,x' \in \R$ we have
{\[
\begin{pmatrix}
u_{n+1}^j(x)\\
u_{n+1}^j(x')
\end{pmatrix}\bigg|u_n \sim N\bigg(\begin{pmatrix}0\\0\end{pmatrix},\begin{pmatrix}h(0) & h\big(\|u_n(x)-u_n(x')\|_2\big)\\h\big(\|u_n(x)-u_n(x')\|_2\big) & h(0)\end{pmatrix}\bigg).
\]}

A common choice of covariance function is the squared exponential kernel:
\begin{align}
\label{eq:squared_exp}
h(z) = \sigma^2 e^{-z^2/2w^2}
\end{align}
where $\sigma^2,w^2 > 0$ are scalar parameters. In \citet{duvenaud2014avoiding}, {in the case $m=d=1$}, the choice $\sigma^2/w^2 = \pi/2$ is made above to ensure that the expected magnitude of the derivative remains constant through iterations. We show in the next proposition that if $\sigma^2, w^2$ are chosen such that {$\sigma^2 < w^2/m$}, then the limiting process is trivial in a sense to be made precise.

%{\bf\color{blue} I've replaced the assumption that $u_0$ is Lipschitz with $u_0$ is bounded on bounded sets, so we don't rule out piecewise constant $u_0$ discussed above.}
\begin{theorem}
\label{thm:triv}
Assume that $h(\cdot)$ is given by the squared exponential kernel (\ref{eq:squared_exp})
and that $u_0$ is bounded on bounded sets almost-surely. {Then if $\sigma^2 < w^2/m$,
\[
\bbP\big(\|u_n(x)-u_n(x')\|_2 \to 0\text{ for all }x,x' \in D\big) = 1
\]}
where $\PP$ denotes the law of the process $\{u_n\}$ over the probability space $\Omega$.
\end{theorem}

\begin{proof}

Since $1-e^{-x} \le x$ for $x \ge 0$ it follows that, for all $z \in \R$,
\[
2h(0) - 2h(z) \leq \frac{\sigma^2}{w^2}z^2,
\]
with equality when $z=0$. Then we have
{\begin{align*}
\mathbb{E}\big(\|u_{n}(x)-u_{n}(x')\|_2^2\big|u_{n-1}\big) &= \sum_{j=1}^m\mathbb{E}\big(|u_{n}^j(x)-u_{n}^j(x')|^2\big|u_{n-1}\big)\\
&=\sum_{j=1}^m \left(2h(0) - 2h\big(\|u_{n-1}(x)-u_{n-1}(x')\|_2\big)\right)\\
&\leq m\frac{\sigma^2}{w^2}\|u_{n-1}(x)-u_{n-1}(x')\|_2^2
\end{align*}
and so using induction and the tower property of conditional expectations,
\begin{align*}
\mathbb{E}\|u_{n}(x)-u_{n}(x')\|_2^2 &\leq \bigg(\frac{m\sigma^2}{w^2}\bigg)\mathbb{E}\|u_{n-1}(x)-u_{n-1}(x')\|_2^2\\
&\leq \bigg(\frac{m\sigma^2}{w^2}\bigg)^{n}\mathbb{E}\|u_0(x)-u_0(x')\|_2^2\\
&\leq \bigg(\frac{m\sigma^2}{w^2}\bigg)^{n} \kappa(x,x')
\end{align*}
for some constant $\kappa(x,x')$. By the Markov inequality, we see that for any $\eps > 0$,
\begin{align}
\label{eq:comp_markov}
\PP\big(\|u_n(x)-u_n(x')\|_2 \geq \eps) \leq \frac{1}{\eps^2}\bigg(\frac{m\sigma^2}{w^2}\bigg)^{n}\kappa(x,x'),
\end{align}
and so applying the first Borel-Cantelli lemma we deduce that
\[
\PP\bigg(\limsup_{n\to\infty} \|u_n(x)-u_n(x')\|_2 \geq \eps\bigg) = 0
\]
since $\sigma^2 < w^2/m$. The above can be rephrased as the statement that for any $\eps > 0$ and any $x,x' \in D$, there exists $\Omega(\eps,x,x')\subseteq\Omega$ with $\PP(\Omega(\eps,x,x')) = 1$ such that for any $\omega \in \Omega(\eps,x,x')$ there exists an $N \in \N$ such that for any $n\geq\N$, $\|u_n(x;\omega)-u_n(x';\omega)\|_2 < \eps$. Let $\{q_j\}$ be a countable dense subset of $D$, and define
\[
\Omega_* = \bigcap_{k,i,j\in\N} \Omega\bigg(\frac{1}{k},q_i,q_j\bigg),
\]
noting that $\PP(\Omega_*) = 1$. Then for any $\omega \in \Omega_*$, $x,x' \in \{q_j\}$ and $\eps > 0$ there exists an $N \in \N$ such that for any $n \geq \N$, $\|u_n(x;\omega)-u_n(x';\omega)\|_2 < \eps$. Since sample paths are almost-surely continuous, the above can be extended to all $x,y \in \R$, so that
\[
\bbP\big(\|u_n(x)-u_n(x')\|_2 \to 0\text{ for all }x,x' \in D\big) = 1.
\]
}
\end{proof}

{
\begin{remark}
\label{rem:comp_F}
\begin{enumerate}
 \item If a more general transformation map $F:\R^m\to D'$ is included, then the above result still holds provided we take $\sigma^2 < w^2/(\|F'\|_\infty m)$. d{The convergence to a constant hence occurs when the length scale $w$ is large or $\|F'\|_\infty$ is small (so each Gaussian random field doesn't change too rapidly across the domain), or when the amplitude $\sigma$ is small (so inputs are not warped too far).}
 \item The condition of the above theorem is less likely to be satisfied as the width $m$ of each layer is increased, and so this triviality pathology is unlikely to arise for large $m$; this may be observed in practice numerically.
 \item Following \citet{neal1995bayesian,duvenaud2014avoiding}, recent works such as \citet{dai2015variational,cutajar2016random} connect all layers to the input layer in order to avoid certain pathologies. The Markovian structure of the process is maintained in this case: with the above notation, the process is then defined by
 \[
 u_{n+1}(x) = \xi_{n+1}(u_n(x),x),\quad \xi_{n+1}^j \sim \mathrm{GP}\big(0,h(\|x-x'\|_2)\big)\;\text{i.i.d},
 \]
 where now $\xi_n:\R^m\times\R^d\to\R^m$. Defining $\beta = m\sigma^2/w^2 < 1$, if $\sigma \geq 1$ we may use the same argument as the proof above to deduce that
\begin{align*}
\mathbb{E}\big(\|u_{n}(x)-u_{n}(x')\|_2^2\big|u_{n-1}\big) \leq \beta\|u_{n-1}(x)-u_{n-1}(x')\|_2^2 + \beta\|x-x'\|_2^2,
\end{align*}
which leads to
\begin{align*}
\mathbb{E}\|u_{n}(x)-u_{n}(x')\|_2^2 &\leq \beta^{n}\mathbb{E}\|u_0(x)-u_0(x')\|_2^2 + \beta\left(\frac{1-\beta^n}{1-\beta}\right)\|x-x'\|_2^2.
\end{align*}
The right hand side does not vanish as $n\to\infty$, and so we can no longer use the first Borel-Cantelli lemma to reach the same conclusion as the case where the layers are not connected to the input layer. This could provide some intuition as to why including the connection of each layer to the input layer provides greater stability than not doing so.
 \end{enumerate}
\end{remark}}

\subsection{Covariance Function}
\label{ssec:F2}

In order to study ergodicity of the deep Gaussian process
defined through covariance functions, we will restrict attention 
in the remainder of this subsection to hierarchies of finite-dimensional 
multivariate Gaussian random variables as in \eqref{eq:pac2}.
Note that although we have here defined $\bu_0 \sim {N}(0, \bR_\mathrm{S})$, following e.g. \citet{paciorek_thesis}, the ergodicity of the deep Gaussian process will be proved for fixed $u_0 \in \R^N$ (cf Theorem \ref{t:ergodic1}).
The following result is immediate from Proposition \ref{lem:posdef}.

\begin{corollary} Let Assumptions \ref{ass:posdef_stat} and \ref{ass:posdef} hold. Then the covariance matrix $\bR(\bu_n)$ is positive definite for all $\bu_n \in C$, for any compact subset of $C \subseteq \R^N$.
\end{corollary}

Note that, because we have chosen to work with a correlation kernel, we have
\begin{equation}
\label{eq:Tr}
{\rm Tr}\bigl(\bR(\bu_{n})\bigr)=N.
\end{equation}
We will use this fact explicitly in the ergodicity proof; however it may be
relaxed as discussed in the Remark \ref{rem:cov} below.

We view the sequence of random variables $\{\bu_n\}_{n=0}^\infty$ as a Markov chain,  
with $u_0 \in \R^N$ given, 
and we want to show the existence of a stationary distribution. Recall
the one-step transition kernel $\sP$ of the Markov chain given by
\eqref{eq:k1}, and its $n-$fold composition given by \eqref{eq:k2}.
In order to prove ergodicity of the Markov chain
we will follow the proof technique in \citet{msh02,meyn_tweedie}, which establishes geometric ergodicity with the following proposition.

\begin{proposition}\label{prop:ergodic} Suppose the Markov chain $\{\bu_{n}\}_{n=0}^\infty$ satisfies, for some compact set $C \in \mathcal B (\R^N)$, the following:
\begin{enumerate}
\item[(i)] For some $y^* \in \text{int}(C)$ and for any $\delta > 0$, we have
\[
\mathsf{P}(u, \mathcal B_\delta(y^*))  > 0 \qquad \text{for all } u \in C.
\]
\item[(ii)] The transition kernel $\mathsf{P}(u, \cdot)$ possesses a density $p(u, y)$ in $C$, precisely
\[
\mathsf{P}(u, A) = \int_A p(u, y)\,\mathrm{d}y, \quad \text{for all } u \in C, \; A \in \mathcal B(\R^N) \cap \mathcal B(C),
\]
and $p(u, y)$ is jointly continuous on $C \times C$. 
\item[(iii)] There is a function $V : \R^N \rightarrow [1,\infty)$, with $\lim_{u \rightarrow \infty} V(u) = \infty$, and real numbers $\alpha \in (0,1)$ and $\beta \in [0,\infty)$ such that
\[
\EE(V(\bu_{n+1}) \, | \, \bu_{n}) \leq \alpha V(\bu_n) + \beta.
\]
\end{enumerate}
If we can choose the compact set $C$ such that
\[
C = \left\{u : V(u) \leq \frac{2\beta}{\gamma - \alpha}\right\},
\]
for some $\gamma \in (\sqrt{\alpha},1)$, then there exists a unique invariant measure $\pi$. Furthermore, there is $r(\gamma) \in (0,1)$ and $\kappa(\gamma) \in (0,\infty)$ such that for all $u_0 \in \R^N$ and all measurable $g$ with $|g(u)| \leq V(u)$ for all $u \in \R^N$, we have
\[
|\mathbb{E}^{\mathsf{P}^n(u_0,\cdot)}(g) - \pi(g)| \leq \kappa r^n V(u_0).
\]
\end{proposition}

We may verify the assumptions of Proposition \ref{prop:ergodic}
leading to the following theorem concerning the ergodicity of
deep Gaussian processes defined via the covariance function:

\begin{theorem}
\label{t:ergodic1}
Suppose Assumptions \ref{ass:posdef_stat} and \ref{ass:posdef} hold. Then the Markov chain $\{\bu_{n}\}_{n=0}^\infty$ satisfies the assumptions of Proposition \ref{prop:ergodic}. As a consequence, there exists $\eps \in (0,1)$ such that for any $u_0 \in \R^N$, there is a $K(u_0) > 0$ with
\[
\|\sP^n(u_0,\cdot)-\pi\|_{TV} \leq K(1-\eps)^n\quad\text{for all }n \in \N,
\]
and so the chain is ergodic.
\end{theorem}

The proof rests on the following three lemmas, and is given
after stating and proving them. {The first lemma shows that, on average, the norm of states of the chain remains constant as the length of the chain is increased. The second shows that, given any current state in $\R^N$ and any ball around the origin in $\R^N$, there is a positive probability that the next state will belong to that ball. The third lemma shows that the probability that the Markov chain moves to a set may be found via integration of a continuous function over that set.}

\begin{lemma}\label{lem:bounded_scaled}{\em (Boundedness)} Suppose Assumptions \ref{ass:posdef_stat} and \ref{ass:posdef} hold. For all $n \in \N$, we have 
\[
\EE \big( \| \bu_{n+1} \|^2_2 \, | \, \bu_n\big) = N.
\]
\end{lemma}
\begin{proof}Let $n \geq 0$. Since the random variable $\bu_{n+1} | \bu_n$ has zero mean, the linearity of expectation implies (using \eqref{eq:Tr}) that
\[
\EE \left( \| \bu_{n+1} \|^2_2 \, | \, \bu_n\right) = \EE \bigg( \sum_{j=1}^N (\bu_{n+1})_j^2 \, \bigg| \, \bu_n \bigg) = \tr(\bR(\bu_n)) = N, 
\]
for all $n \in \N$. %The equality $\EE \left( \| \bu_{0} \|^2_2 \right) = N$ follows similarly. 
\end{proof}

\begin{lemma}\label{lem:ball_scaled}{\em (Positive probability of a ball around zero)} Suppose Assumptions \ref{ass:posdef_stat} and \ref{ass:posdef} hold. For all $u \in \R^N$ and $\delta > 0$, we have
\[
\mathsf{P}\big(u, \mathcal B_\delta(0) \big) > 0.
\]
\end{lemma}
\begin{proof} We have the equality $\bu_{n+1}|(\bu_n = u) = \sqrt{\bR(u)} \xi_{n+1}$ in distribution, where $\sqrt{\bR(u)}$ denotes the Cholesky factor of the correlation matrix $\bR(u)$ and $\xi_{n+1} \sim \mathrm{N}(0, \mathrm{I}_N)$. Then
\begin{align*}
\mathsf{P}\big(u, \mathcal B_\delta(0) \big) &= \PP\big(\|\bu_{n}\|_2 \leq \delta \, | \, \bu_{n-1} = u\big) \\
&= \PP\Big(\big\|\sqrt{\bR(u)} \xi_{n+1}\big\|_2 \leq \delta \Big) \\
&\geq \PP\Big(\big\|\sqrt{\bR(u)}\big\|_2 \big\| \xi_{n+1}\big\|_2 \leq \delta \Big) \\
&= \PP\Big( \big\| \xi_{n+1}\big\|_2 \leq \delta \big\|\sqrt{\bR(u)}\big\|_2^{-1} \Big).
\end{align*}
To show that the latter probability is positive, we need to show that $\delta \big\|\sqrt{\bR(u)}\big\|_2^{-1} > 0$. Since $\delta >0$ is fixed, we only need to show $\big\|\sqrt{\bR(u)}\big\|_2 < \infty$. Since $\big\|\sqrt{\bR(u)}\big\|_2^2 = \rho(\bR(u))$, the spectral radius of $\bR(u)$, we have
\[
\big\|\sqrt{\bR(u)}\big\|_2^2 = \rho(\bR(u)) \leq \tr(\bR(u)) = N.
\]
The claim then follows.
\end{proof}

\begin{lemma}\label{lem:density_scaled}{\em (Transition probability has a density)} Suppose Assumptions \ref{ass:posdef_stat} and \ref{ass:posdef} hold. Then the transition probability $\mathsf{P}\big(u, \cdot\big)$ has a jointly continuous density $p(u,y)$ for all $u \in C$, for any compact set $C \subseteq \R^N$.
\end{lemma}
\begin{proof} We have $\bu_{n+1} | (\bu_{n} = u) \sim N(0, \bR(u))$, and the existence of a jointly continuous density of the transition probability in $C$ follows if $\bR(u)$ is positive definite for all $u \in C$. The claim then follows by Proposition \ref{lem:posdef}.
\end{proof}

We may now use the three preceding lemmas to prove the main ergodic theorem
for deep Gaussian processes defined through the covariance function.

\begin{proof}[Proof of Theorem \ref{t:ergodic1}] Lemma \ref{lem:ball_scaled} shows that assumption (i) is satisfied, for any $C$ containing $y^*=0$, and Lemma \ref{lem:density_scaled} shows that assumption (ii) is satisfied, for any compact set $C$. It follows from Lemma \ref{lem:bounded_scaled} that assumption (iii) is satisfied, with $V(u) = \|u\|_2^2 + 1$, any $\alpha \in (0,1)$ and $\beta = N+1$.
Now choose $\alpha = 1/4$ and $\gamma = 3/4 \in (\sqrt{\alpha},1)$, so that the set
\[
C = \left\{u : V(u) \leq \frac{2\beta}{\gamma - \alpha}\right\} = \big\{u : \|u\|_2^2 \leq 4N + 4\big\}
\]
is compact. Then there is a unique invariant measure $\pi$, and there is $r(\gamma) \in (0,1)$ and $\kappa(\gamma) \in (0,\infty)$ such that for $u_0 \in \R^N$ and all measurable $g$ with $|g(u)| \leq V(u)$ for all $u \in \R^N$, we have
\begin{align}
\label{eq:ergodic_tv}
|\mathbb{E}^{\mathsf{P}^n(u_0,\cdot)}(g) - \pi(g)| \leq \kappa r^n V(u_0).
\end{align}
Since $V(u) \geq 1$ for all $u \in \R^N$, the above holds in particular for all measurable $g$ with $\|g\|_\infty \leq 1$. Taking the supremum over all such $g$ in (\ref{eq:ergodic_tv}) yields the given total variation bound, with $K= \kappa V(u_0)$ and $\eps = 1-r$.
\end{proof}

\begin{remark}\label{rem:cov}(Covariance vs correlation kernels) In this 
subsection we have restricted our attention to correlation kernels $\rho_\mathrm{S}(\|x_i-x_j\|_2)$ and $\rho(x_i, x_j;u_n)$, rather than more general covariance kernels $c_\mathrm{S}(\|x_i-x_j\|_2) = \sigma^2_\mathrm{S}\rho_\mathrm{S}(\|x_i-x_j\|_2)$ and $c(x_i, x_j;u_n) = \sigma(x_i;u_n) \sigma(x_j;u_n) \rho(x_i, x_j;u_n)$, for stationary and non-stationary marginal standard deviation functions $\sigma_\mathrm{S} \in (0,\infty)$ and $\sigma: \R^d \rightarrow (0,\infty)$ respectively. This restriction is solely for ease of presentation; 
the analysis presented readily extends to $c(x_i, x_j;u_n)$, 
under suitable assumptions on $\sigma$. 
In  particular the analysis may be
adapted to the case of general covariance kernels 
$c_\mathrm{S}(\|x_i-x_j\|_2)$ and $c(x_i, x_j;u_n)$ under the 
assumption that there exist positive constants $\sigma^-,\sigma^+$ 
such that $\sigma^- \leq \sigma(z) \leq \sigma^+$, for all $z \in \R^d$. 
When general covariances are used then it is possible to ensure
that every multivariate Gaussian random variable in the hierarchy is 
of the same amplitude by scaling the corresponding 
covariance matrix $\bC(\bu_n)$ to have constant trace $N$ 
at each iteration $n$; the average variance over all points
$\{x_i\}_{i=1}^N$ is then $1$ for every $n$.
\end{remark}

\subsection{Covariance Operator}
\label{ssec:O4}

We consider the class of covariance operators introduced in section \ref{ssec:O3} and show that, under precise assumptions detailed below, 
the iteration (\ref{eq:cg3}) produces an ergodic Markov chain. Unlike
the previous subsection, where we worked on $\R^N$, here we will work on 
the  separable Hilbert space $\cH= L^2(D;\R)$. 
To begin with, define the precision operators (densely
defined on $\cH$ \citet{HVS,pinski2015kullback}),
\begin{align*}
C_-^{-1} &= P,\\
C_+^{-1} &= P + F_+I,\\
C(u)^{-1} &= P + \Gamma(u),\quad u \in \cH,
\end{align*}
and the probability measures
\begin{align*}
\mu_- &= N(0,C_-),\\
\mu_+ &= N(0,C_+),\\
\mu(\cdot\,;u) &= N(0,C(u)),\quad u \in \cH.
\end{align*}
Throughout the rest of this section we make the following assumptions on $C_-$ and $F$:
\begin{assumptions}
\label{ass:op}
\begin{enumerate}
%\item The Hilbert space $(\cH,\langle\cdot,\cdot\rangle)$ is separable.
\item The operator $C_-:\cH\to\cH$ is symmetric and positive, and its eigenvalues $\{\lambda_j^2\}$ have algebraic decay $\lambda_j^2 ymp j^{-r}$ for some $r > 1$.
\item The function $F:\cH\to\R$ is continuous, and there exists $F_+ \geq 0$ such that $0 \leq F(u) \leq F_+$ for all $u \in \cH$.
\end{enumerate}
\end{assumptions}

\begin{remark}
\begin{enumerate}
\item The assumption on algebraic decay of the eigenvalues can be relaxed to the operator $C_-$ being trace-class on $\cH$; however the arguments that follow are cleaner when we assume this explicit decay which, of course, implies the trace
condition.
Note also that, under the stated assumption on algebraic decay, Gaussian measures
on $L^2(D;\R)$ will be supported on $X=C(D;\R)$ under mild conditions
on the eigenfunctions of $C_-$ \citet{stuart2010inverse} 
so that $F\bigl(u(x)\bigr)$ will be defined for all $x \in D$ rather
than $x$ a.e. in $D.$ Then $\Gamma(u)v$ makes sense pointwise when $v \in X$. 
{\item The assumed form of the precision operator together with Assumptions \ref{ass:op} mean that the resulting family of measures $\{\mu(\cdot;u)\}_{u\in\cH}$ will be mutually equivalent. This allows for the total variation metric between measures to be used, and a concise proof of ergodicity to be obtained. If the measures were singular, a different metric such as the Wasserstein metric would be required to quantify the convergence.}
\end{enumerate}
\end{remark}

We now prove the following ergodic theorem for the deep Gaussian
processes constructed through covariance operators.

\begin{theorem}
\label{t:added}
Let Assumptions \ref{ass:op} hold, and let the Markov chain $\{u_n\}$ be given by (\ref{eq:cg3}) with $L(u) = C(u)^{\frac{1}{2}}$ as defined above. Then there exists a unique invariant distribution $\pi$, and there exists $\eps > 0$ such that for any $u_0 \in \cH$,
\[
\|\sP^n(u_0,\cdot)-\pi\|_{TV} \leq (1-\eps)^n\quad\text{for all }n \in \N.
\]
In particular, the chain is ergodic.
\end{theorem}

The following lemma will be used to show a minorization condition, as well as establish further notation, key to the proof of Theorem \ref{t:added}
which follows it. {It essentially shows a stronger form of equivalence of the family of measures $\{\mu(\cdot,u)\}_{u \in \cH}$.}

\begin{lemma}
\label{lem:minor_op}
Let Assumptions \ref{ass:op} hold. Then there exists $\eps > 0$ such that for any $u,v \in \cH$,
\[
\frac{\dee \mu(\cdot\,;u)}{\dee \mu_+}(v) \geq \eps.
\]
\end{lemma}

\begin{proof}

The assumptions on $F$ mean that the measures $\mu(\cdot\,;u)$, $\mu_-$ and $\mu_+$ are mutually absolutely continuous, with
\begin{align*}
\frac{\dee \mu(\cdot\,;u)}{\dee \mu_-}(v) &= \frac{1}{Z(u)}\exp\bigg(-\frac{1}{2}\big\langle v,F(u)v\big\rangle\bigg),\\
Z(u) &= \mathbb{E}^{\mu_-}\bigg[\exp\bigg(-\frac{1}{2}\big\langle v,F(u)v\big\rangle\bigg)\bigg];\\
\frac{\dee \mu_+}{\dee \mu_-}(v) &= \frac{1}{Z_+}\exp\bigg(-\frac{1}{2}\big\langle v,F_+ v\big\rangle\bigg),\\
Z_+ &= \mathbb{E}^{\mu_-}\bigg[\exp\bigg(-\frac{1}{2}\big\langle v,F_+ v\big\rangle\bigg)\bigg].
\end{align*}
Observe that we may bound $Z(u) \leq 1$ uniformly in $u\in H$ since $F \geq 0$. Additionally, we have that
\[
Z_+ \geq \mathbb{E}^{\mu_-}\bigg[\exp\bigg(-\frac{1}{2}\big\langle v,F_+ v\big\rangle\bigg)\mathds{1}_{\|v\|^2 \leq 1}\bigg] \geq \exp\bigg(-\frac{1}{2}F_+\bigg)\mu_-\big(\|v\|^2 \leq 1\big) =: \eps > 0.
\]
Note that $\eps$ is positive since $\cH$ is separable, and thus all balls 
have positive measure \citet{hairerSPDE}. It follows that
\begin{align*}
\frac{\dee \mu(\cdot\,;u)}{\dee \mu_+}(v) &= \frac{\dee \mu(\cdot\,;u)}{\dee \mu_-}(v)\times \bigg(\frac{\dee \mu_+}{\dee \mu_-}(v)\bigg)^{-1}\\
&= \frac{1}{Z(u)}\exp\bigg(-\frac{1}{2}\big\langle v,F(u)v\big\rangle\bigg) \times Z_+ \exp\bigg(\frac{1}{2}\big\langle v,F_+ v\big\rangle\bigg)\\
&\geq \eps\exp\bigg(\frac{1}{2}\big\langle v,\big(F_+-F(u)\big)v\big\rangle\bigg)\\
&\geq \eps
\end{align*}
since $F_+$ bounds $F$ above uniformly. 
\end{proof}

\begin{proof}[Proof of Theorem \ref{t:added}]
We first establish existence of at least one invariant distribution by showing that chain $\{u_n\}$ is (strong) Feller, and that for each $u_0 \in \cH$ the family $\{\sP^n(u_0,\cdot)\}$ of transition kernels is tight.  To see the former, let $f:\cH\to\R$ be any bounded measurable function. We have that, for any $v \in \cH$,
\begin{align*}
(\sP f)(u) &:= \int_{\cH} f(v)\sP(u,\dee v)\\
&= \int_{\cH}f(v)\frac{1}{Z(u)}\exp\bigg(-\frac{1}{2}\langle v,F(u)v\rangle\bigg)\,\mu_-(\dee v).
\end{align*}
Since $F(u) \leq F_+$ it follows that $Z(u)$ is bounded below by a positive constant, 
uniformly with respect to $u$. Additionally $F$ is continuous and non-negative, and so the integrand is bounded and continuous with respect to $u$. Hence given any sequence $u^{(k)}\to u$ in $\cH$, we may apply the dominated convergence theorem to see that $(\sP f)(u^{(k)})\to(\sP f)(u)$. The function $\sP f$ is therefore continuous, and so the chain $\{u_n\}$ is strong Feller.

We now show tightness. The assumptions on the operator $C_-:\mathcal{H}\to\mathcal{H}$ imply that it is trace-class, and so in particular compact. It is also positive and symmetric, and so by the spectral theorem, admits a complete orthonormal system of eigenvectors $\{\varphi_j\}$ with corresponding positive eigenvalues $\{\lambda_j^2\}$ such that $\lambda_j^2 \to 0$. Given $s > 0$, define the subspace $\cH^s\subset\cH$ by
\[
\cH^s = \bigg\{v \in \mathcal{H}\,\bigg|\,\|v\|_{\cH^s}^2 := \sum_{j=1}^\infty j^{2s}|\langle\varphi_j,v\rangle|^2 < \infty\bigg\}.
\]
It is standard to show that $\cH^s$ is compactly embedded in $\cH$ for any $s > 0$, see for example Appendix A.2 in \citet{robinson2001infinite}. By the Karhunen-Lo\'eve theorem, any $v \sim \mu_-$ may be represented as
\[
v = \sum_{j=1}^\infty \lambda_j\xi_j\varphi_j,\quad\xi_j\sim N(0,1)\text{ i.i.d.}
\]
Hence, by the orthonormality of the $\{\varphi_j\}$ and the assumed decay of the eigenvalues, we have that
\[
\mathbb{E}^{\mu_-}\big(\|v\|_{\mathcal{H}^s}^2\big) = \sum_{j=1}^\infty j^{2s}\lambda_j^2 ymp \sum_{j=1}^\infty j^{2s-r}
\]
and so
\[
\mathbb{E}^{\mu_-}\big(\|v\|_{\mathcal{H}^s}^2\big) < \infty\quad\text{if and only if}\quad s < \frac{r}{2} - \frac{1}{2}.
\]
Since $r > 1$ by assumption, we can always choose $s > 0$ such that this holds; fix such an $s$ in what follows. Observe that, for any $n \in \N$,
\begin{align*}
\mathbb{E}\big(\|u_n\|_{\cH^s}^2) &= \mathbb{E}\big(\mathbb{E}\big(\|u_n\|_{\cH^s}^2\big| u_{n-1}\big)\big)\\
&= \mathbb{E}\big(\mathbb{E}^{\mu(\cdot\,;u_{n-1})}\big(\|v\|_{\cH^s}^2\big)\big)\\
&= \mathbb{E}\bigg(\int_{\cH} \|v\|_{\cH^s}^2\frac{1}{Z(u_{n-1})}\exp\bigg(-\frac{1}{2}\big\langle v,F(u_{n-1})v\big\rangle\bigg)\,\mu_-(\dee v)\bigg)\\
&\leq \frac{1}{Z_+}\mathbb{E}^{\mu_-}\big(\|v\|_{\cH^s}^2\big)\\
&=: M < \infty.
\end{align*}
We have bounded $Z(u_{n-1}) \geq Z_+$ using that $F(u_{n-1}) \leq F_+$. Applying the Chebychev inequality, we have for each $n \in N$ and $R > 0$
\begin{align*}
\mathbb{P}\big(\|u_n\|_{\cH^s} > R\big) \leq \frac{\mathbb{E}\big(\|u_n\|_{\cH^s}^2\big)}{R^2} \leq \frac{M}{R^2},
\end{align*}
and so given any $\kappa > 0$,
\[
\mathbb{P}\bigg(\|u_n\|_{\cH^s} \leq \sqrt{\frac{M}{\kappa}}\bigg) \geq 1-\kappa.
\]
This can be rewritten as
\[
\mathsf{P}^n(u_0,K_\kappa) \geq 1-\kappa
\]
where $K_\kappa = \big\{u \in \cH\,|\,\|u\|_{\cH^s} \leq \sqrt{M/\kappa}\big\}$ is compact in $\cH$, since $\cH^s$ is compactly embedded in $\cH$; this shows tightness of the sequence of probability measures $\mathsf{P}^n(u_0,\cdot)$. Since tightness implies boundedness in probability on average, an application of Theorem 12.0.1 in \citet{meyn_tweedie} gives existence of an invariant distribution.

Lemma \ref{lem:minor_op} shows that $\{u_n\}$ satisfies a global minorization condition for the one-step transition probabilities: for any $u_0 \in \cH$ and any measurable $A\subseteq\cH$,
\[
\mathsf{P}\big(u_0,A\big) = \mathbb{E}^{\mu(\cdot\,;u_0)}\big(\mathds{1}_A(v)\big) = \mathbb{E}^{\mu_+}\bigg(\frac{\dee \mu(\cdot\,;u_0)}{\dee \mu_+}(v)\mathds{1}_A(v)\bigg) \geq \eps \mu_+(A).
\]
Combined with the existence of an invariant distribution above, a short coupling argument (Theorem 16.2.4 in \citet{meyn_tweedie}) gives the result with the same $\eps$ as above.
\end{proof}

\subsection{Convolution}
\label{ssec:conv_proof}

The convolution iteration has the advantage that, through use of Fourier series and 
the law of large numbers, its long time 
behaviour can be completely characterized analytically. 
We consider the convolution as a random map on $\cH = L^2(D;\mathbb{C})$, 
$D = (0,1)^d$. The iteration is given by
\begin{align}
\label{eq:convolution}
u_{n+1}(x) = (u_n * \xi_{n+1})(x) := \int_{D} u_n(x-y)\xi_{n+1}(y)\,dy,\quad \xi_{n+1}\sim N(0,C)\text{ i.i.d.}
\end{align}
where we implicitly work with periodic extensions to define the convolution. We assume that $C$ is a negative fractional power of a differential operator so that it diagonalizes in Fourier space; such a form of covariance operator is common in applications, d{as it includes, for example, Whittle-Mat\'ern distributions \citet{matern_spde}}. For example, we may take
\[
C = (I - \Delta)^{-\alpha},\quad D(-\Delta) = H^2_{\mathrm{per}}([0,1]^d)\subset \cH,
\]
in which case the samples $\xi_{n+1} \sim N(0,C)$ will (almost surely) possess $s$ fractional Sobolev and H\"older derivatives for any $s < \alpha - d/2$; see
\citet{stuart2010inverse} for details.

We choose the orthonormal Fourier basis
\[
\varphi_k(x) = e^{2\pi i k\cdot x},\quad k \in \mathbb{Z}^d
\]
which are the eigenvectors of $C$; we denote the corresponding eigenvalues $\{\lambda_k^2\}$. Given $u \in \cH$ and $k \in \mathbb{Z}^d$, define the Fourier coefficient $\hat{u}(k)\in\mathbb{C}$ by
\[
\hat{u}(k) := \langle\varphi_k,u\rangle_{L^2} = \int_D \overline{\varphi_k(x)}u(x)\,\dee x.
\]
Then it can be readily checked that for any $u,v \in \cH$ and $k \in \mathbb{Z}^d$,
\begin{align}
\label{eq:fconv}
\widehat{(u*v)}(k) = \hat{u}(k)\hat{v}(k).
\end{align}
We use this property to establish the following theorem.

\begin{theorem}
\label{T:convolution}
Let $C:\cH\to\cH$ be a negative fractional power of a differential operator such that $C$ is positive, symmetric and trace-class, with eigenvectors $\{\psi_k\}$ and eigenvalues $\{\lambda_k^2\}$. Define the Markov chain $\{u_n\}$ by (\ref{eq:convolution}). Then for any $u_0 \in \cH$,
\[
\lim_{n\rightarrow\infty} |\hat{u}_n(k)|^2 =
\begin{cases}
0 & |\lambda_k|^2 < 2e^\gamma\\
\infty & |\lambda_k|^2 > 2e^\gamma
\end{cases}\quad\text{almost surely}
\]
where $\gamma\approx 0.577$ is the Euler-Mascheroni constant. In particular, if $|\lambda_k|^2 < 2e^\gamma$ for all $k \in \mathbb{Z}^d$, then 
every Fourier coefficient of $u_n$ tends to zero almost surely
and hence $u_n\rightharpoonup 0$ in $\cH$ almost surely.
\end{theorem}

\begin{proof}
First observe that by the Karhunen-Lo\'eve theorem, we may express $\xi_{n+1} \sim N(0,C)$ as
\[
\xi_{n+1} = \sum_{k\in\mathbb{Z}^d} \lambda_k\eta_{n,k}\varphi_k,\quad \eta_{n,k}\sim N(0,1)\text{ i.i.d.}
\]
and so, since $\{\varphi_k\}$ is orthonormal,
\[
\hat{\xi}_{n+1}(k) = \lambda_k\eta_{n,k}.
\]
Then by the property (\ref{eq:fconv}), we see that for each $k \in \mathbb{Z}^d$ and $n \in \N$,
\begin{equation}
\label{eq:conv_iter2}
\hat{u}_{n+1}(k) = \hat{u}_n(k)\hat{\xi}_{n+1}(k) = \hat{u}_n(k)\lambda_k\eta_{n,k}
\end{equation}
where the second equality is in distribution. The problem has now been reduced to an independent family of scalar problems. 
We can write $\hat{u}_n(k)$ explicitly as
\begin{align}
\label{eq:conv_iter}
\hat{u}_n(k) = \hat{u}_0(k) \prod_{j=1}^n \lambda_k\eta_{j,k}.
\end{align}
Now observe that
\begin{align}
\label{eq:slln}
\notag|\hat{u}_n(k)|^2 &= |\hat{u}_0(k)|^2 \prod_{j=1}^n |\lambda_k|^2|\eta_{j,k}|^2\\
\notag&= |\hat{u}_0(k)|^2\exp\left(n\cdot\frac{1}{n}\sum_{j=1}^n \log\big(|\lambda_k|^2|\eta_{j,k}|^2\big)\right)\\
&= |\hat{u}_0(k)|^2\exp\left(n\cdot\left(\frac{1}{n}\sum_{j=1}^n \log|\eta_{j,k}|^2 + \log|\lambda_k|^2\right)\right).
\end{align}
By the strong law of large numbers, the scaled sum inside the exponential converges almost surely to $\mathbb{E}(\log|\eta_{1,k}|^2)$. This can be calculated as 
\[
\mathbb{E}(\log|\eta_{1,k}|^2) = -\gamma - \log2.
\]
If the bracketed term inside the exponential in (\ref{eq:slln}) is eventually negative almost surely, then the limit of $|u_n(k)|^2$ will be zero almost surely. This is guaranteed when $-\gamma - \log 2 + \log|\lambda_k|^2 < 0$, i.e. $|\lambda_k|^2 < 2e^\gamma$. Similarly we get divergence if the bracketed term is eventually positive, which happens when $|\lambda_k|^2 > 2e^\gamma$.
\end{proof}

\begin{remark}
It is interesting to note that we may take expectations in 
(\ref{eq:conv_iter2}) to establish that
\[
\mathbb{E}|\hat{u}_n(k)|^2 = |\hat{u}_0(k)|^2 |\lambda_k|^{2n}
\]
and so
\[
\lim_{n\rightarrow\infty} \mathbb{E}|\hat{u}_n(k)|^2 =
\begin{cases}
0 & |\lambda_k|^2 < 1\\
\infty & |\lambda_k|^2 > 1
\end{cases}.
\]
In particular, if $|\lambda_k|^2 \in (1,2e^\gamma)$, then $|\hat{u}_n(k)|^2$ converges to zero almost surely, but diverges in mean square.
\end{remark}

Via a slight modification of the above proof to account for different boundary conditions, we have the following result.
\begin{corollary}
Let $D = (0,1)$ and let $\{u_n\}$ be defined by the iteration (\ref{eq:convolution}), where each $\xi_{n+1}$ is a Brownian bridge. Then $u_n\rightharpoonup 0$ almost surely.
\end{corollary}

\begin{proof}
The  Brownian bridge on $[0,1]$ has covariance operator $(-\Delta)^{-1}$, where
\[
D(-\Delta) = \{u \in H^2_{\mathrm{per}}([0,1])\,|\,u(0) = u(1) = 0\}.
\]
The result of Theorem \ref{T:convolution} cannot be applied directly, since the basis functions $\{\varphi_k\}$ do not satisfy the boundary conditions. The eigenfunctions with the correct boundary conditions are given by
\[
\psi_j(x) = \sqrt{2}\sin(j\pi x) = \frac{1}{\sqrt{2}i}\big(\varphi_j(x) - \varphi_{-j}(x)\big),\quad j \geq 1
\]
with corresponding eigenvalues $\alpha_j^2 = (\pi^2j^2)^{-1}$. A Brownian bridge $\xi_{n+1} \sim N(0,(-\Delta)^{-1})$ can then be expressed as
\[
\xi_{n+1} = \sum_{j=1}^\infty \alpha_j\zeta_{n,j}\psi_j,\quad \zeta_{n,j}\sim N(0,1)\text{ i.i.d.}
\]
by the Karhunen-Lo\'eve theorem. We calculate
\begin{align*}
\hat{u}_{n+1}(k) &= \hat{u}_n(k)\hat{\xi}_n(k)\\
&= \hat{u}_n(k)\sum_{j=1}^\infty \alpha_j\zeta_{n,j}\langle \varphi_k,\psi_j\rangle\\
&= \hat{u}_n(k)\sum_{j=1}^\infty \alpha_j\zeta_{n,j}\frac{1}{\sqrt{2}i}\big(\langle \varphi_k,\varphi_j\rangle - \langle \varphi_k,\varphi_{-j}\rangle\big)\\
&= \hat{u}_n(k)\frac{\mathrm{sgn}(k)\alpha_{|k|}}{\sqrt{2}i}\zeta_{n,|k|}\\
&= \hat{u}_n(k)\lambda_k\eta_{n,k},\quad \eta_{n,k}\sim N(0,1).
\end{align*}
We can now proceed as in Theorem \ref{T:convolution} to deduce that $|u_n(k)|^2\to 0$ whenever $|\lambda_k|^2 < 2e^\gamma$; note that the correlations between $\hat{u}_{n}(k)$ and $\hat{u}_n(-k)$ do not affect the argument. Now observe that $|\lambda_k|^2 = (2\pi^2k^2)^{-1} < 1 < 2e^\gamma$ for all $k$, and the result follows.
\end{proof}

\begin{remark}
The preceding results also holds if we replace the Brownian bridge
by a Gaussian process with precision operator the negative Laplacian
subject to Neumann boundary conditions and spatial mean zero; 
the eigenfunctions are then 
\[
\psi_j(x) = \sqrt{2}\cos(j\pi x) = \frac{1}{\sqrt{2}}\big(\varphi_j(x) + \varphi_{-j}(x)\big),\quad j\geq 1.
\]
The argument is identical, except no $\mathrm{sgn}(k)$ term appears in $\lambda_k$.
\end{remark}

\section{Numerical Illustrations}
\label{sec:I2}

We now study two of the constructions of deep Gaussian processes numerically. In subsection \ref{ssec:cov_fun} we look at realizations of the deep Gaussian process constructed using the covariance function formulation, and in subsection \ref{ssec:cov_op} we perform similar experiments for the covariance operator formulations. Finally we consider Bayesian inverse problems, in which we choose deep Gaussian processes as our prior distributions;
we introduce a function space MCMC algorithm, which scales well under
mesh refinement of the functions to be inferred, for sampling.

For the composition construction, numerical experiments can be found in, for example, \citet{damianou2013deep,duvenaud2014avoiding}. We do not provide numerical experiments for the convolution construction; Theorem \ref{T:convolution} tells us that interesting behaviour cannot be expected in this case.

\subsection{Covariance Function}
\label{ssec:cov_fun}
We start by investigating typical realizations of a deep Gaussian process, constructed through anisotropic covariance kernels as in section \ref{ssec:F}. As the basis of our construction, we choose a stationary Gaussian correlation kernel, given by
\[
\rho_\mathrm{S}(r) = \exp(-r^2),  \qquad r>0.
\]
The function $F$ determining the length scale of the kernel $\rho(\cdot, \cdot; u_n)$ is chosen as $F(x) = x^2$, such that $\Sigma(z) = \left(u_{n}(z)\right)^2 \, \mathrm{I}_d$. {Similar results are obtained with other choices of $F$ in terms of the distribution of samples $u_n$. The choice of $F$ does, however, influence the conditioning of the correlation matrix $\mathbf{R}(\mathbf u_n)$, and the choice $F(x) = \exp(x)$, for example, can lead to numerical instabilities.} As described in section \ref{ssec:F}, we will sample from the finite dimensional distributions obtained by sampling from the Gaussian process at a finite number of points in the domain $D$. To generate the samples, we use the command \texttt{mvnrnd} in MATLAB, and when plotting the samples, we use linear interpolation.

In Figure \ref{f:prior_1d_ker}, we show four independent realizations of the first seven layers $u_0, \dots, u_6$, where $u_0$ is taken as a sample of the stationary Gaussian process with correlation kernel $\rho_\mathrm{S}$. The domain $D$ is here chosen as the interval $(0,1)$, and the sampling points are given by the uniform grid $x_i = \frac{i-1}{256}$, for $i=1, \dots, 257$. Each column in Figure \ref{f:prior_1d_ker} corresponds to one realization, and each row corresponds to a given layer $u_n$, the first row showing $u_0$. We can clearly see the non-stationary behaviour in the samples when progressing through the levels. We note that the ergodicity of the chain is also reflected in the samples, with the distribution of the samples $u_n$ looking similar for larger values of $n$.  

Figure \ref{f:prior_2d_ker} shows the same information as Figure \ref{f:prior_1d_ker}, in the case where the domain $D$ is $(0,1)^2$ and the sampling points are the tensor product of the one-dimensional points $x^1_i = \frac{i-1}{64}$, for $i=1, \dots, 65$.

\begin{figure}
\begin{center}
\begin{overpic}[width=\linewidth,trim=1cm 1.5cm 2cm 1.5cm,clip]{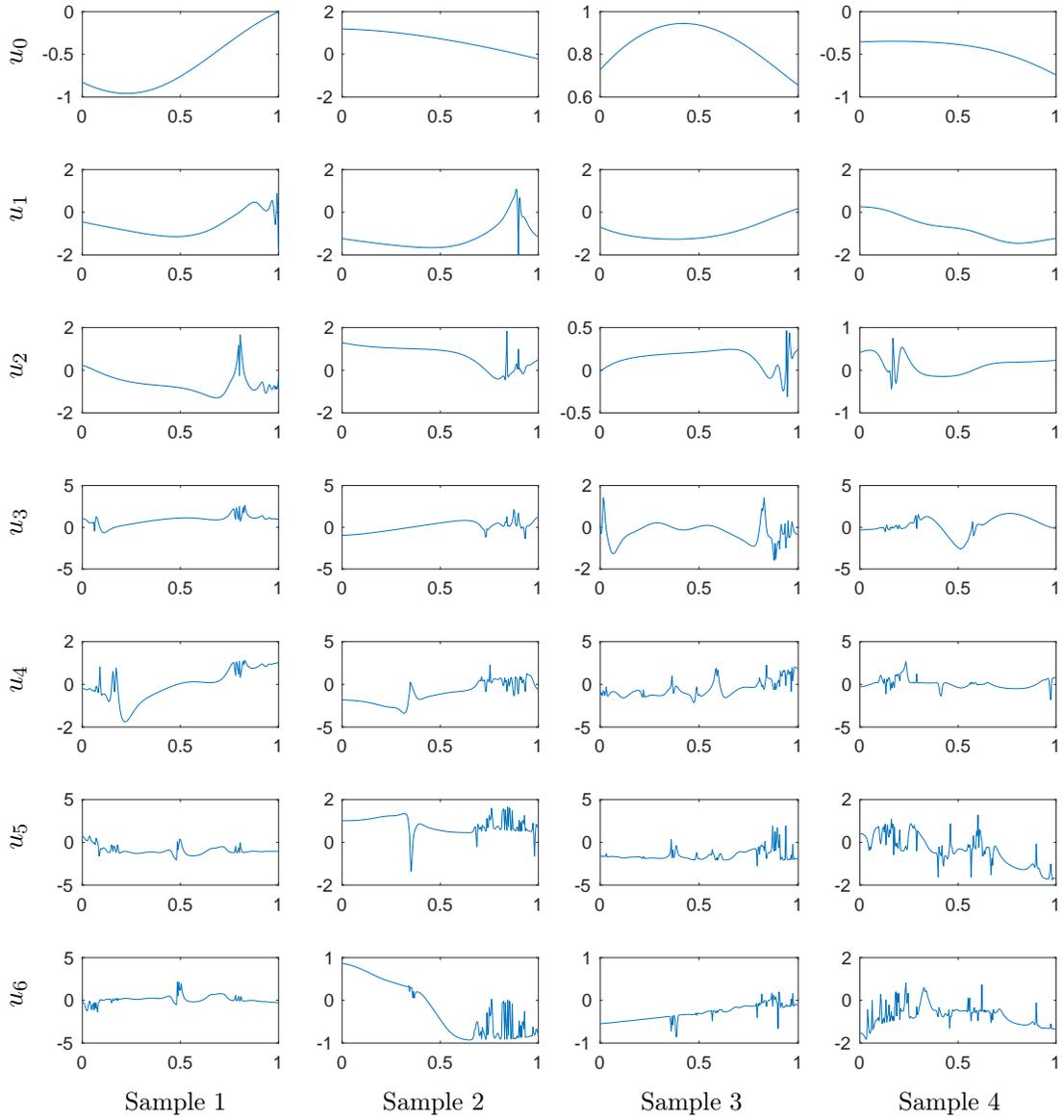}
\put(3.5,93.2){\rotatebox{90}{$u_0$}}
\put(3.5,79.5){\rotatebox{90}{$u_1$}}
\put(3.5,65.8){\rotatebox{90}{$u_2$}}
\put(3.5,52.1){\rotatebox{90}{$u_3$}}
\put(3.5,38.4){\rotatebox{90}{$u_4$}}
\put(3.5,24.7){\rotatebox{90}{$u_5$}}
\put(3.5,11.0){\rotatebox{90}{$u_6$}}

\put(13.5,2){\small Sample 1}
\put(36,2){\small Sample 2}
\put(58.5,2){\small Sample 3}
\put(81,2){\small Sample 4}
\end{overpic}
\end{center}
\caption{Four independent realizations of the first seven layers of a deep Gaussian process, in one spatial dimension, using the covariance kernel construction described in subsection \ref{ssec:F}. Each column corresponds to an independent chain, and layers $u_0,u_1,\ldots,u_6$ are shown from top-to-bottom.}
\label{f:prior_1d_ker}
\end{figure}

\begin{figure}
\begin{center}
\begin{overpic}[width=\linewidth,trim=1.5cm 1.7cm 1.7cm 0cm,clip]{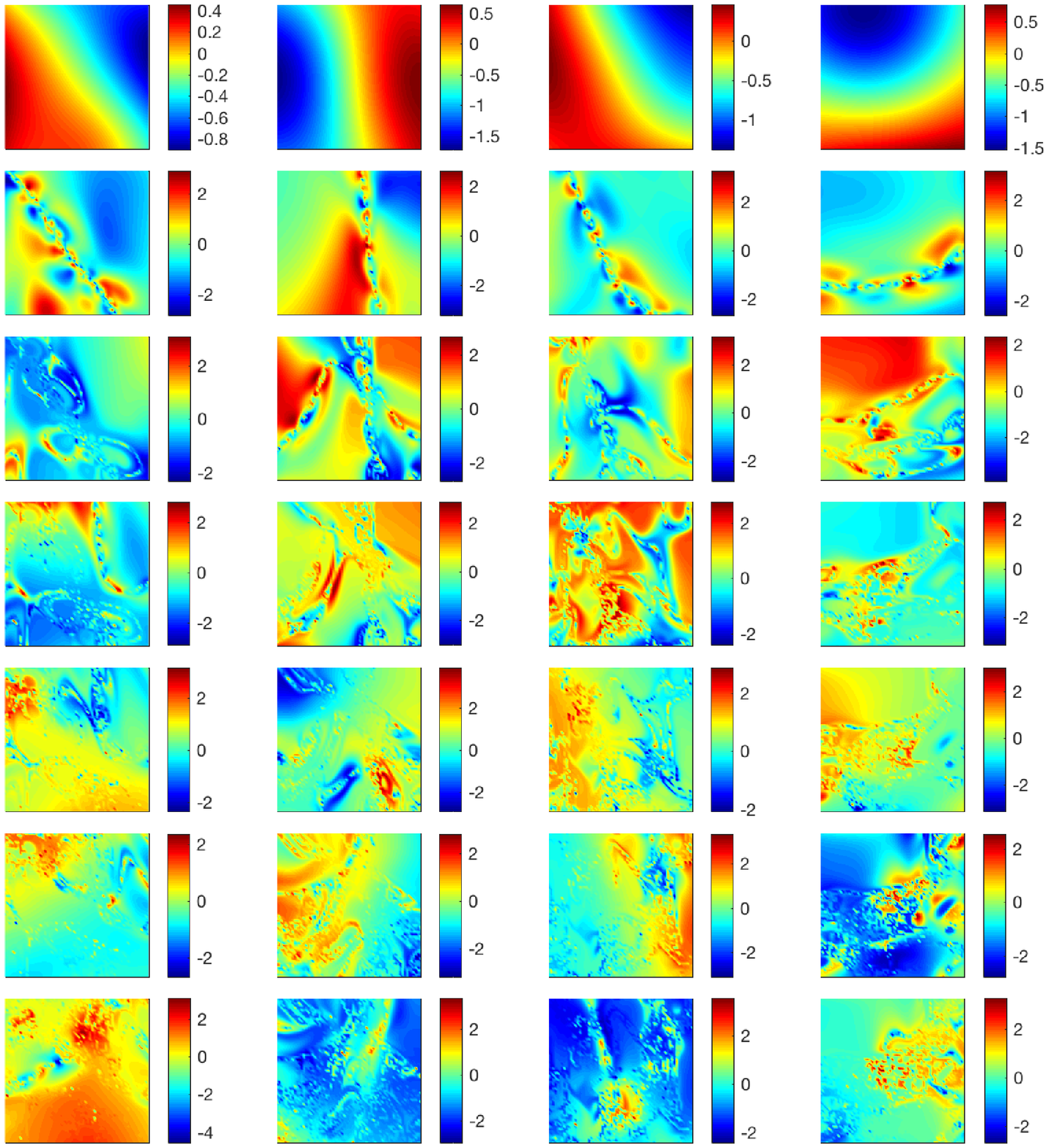}
\put(4,89){\rotatebox{90}{$u_0$}}
\put(4,76.2){\rotatebox{90}{$u_1$}}
\put(4,63.3){\rotatebox{90}{$u_2$}}
\put(4,50.4){\rotatebox{90}{$u_3$}}
\put(4,37.5){\rotatebox{90}{$u_4$}}
\put(4,24.6){\rotatebox{90}{$u_5$}}
\put(4,11.7){\rotatebox{90}{$u_6$}}

\put(11,2){\small Sample 1}
\put(32.2,2){\small Sample 2}
\put(53.4,2){\small Sample 3}
\put(74.6,2){\small Sample 4}
\end{overpic}
\end{center}
\caption{Four independent realizations of the first seven layers of a deep Gaussian process, in two spatial dimensions, using the covariance kernel construction described in subsection \ref{ssec:F}. Each column corresponds to an independent chain, and layers $u_0,u_1,\ldots,u_6$ are shown from top-to-bottom.}
\label{f:prior_2d_ker}
\end{figure}

\subsection{Covariance Operator}
\label{ssec:cov_op}

We now consider the covariance operator construction of the deep Gaussian process. In order to produce more interesting behaviour in the samples, we move away from the absolutely continuous setting considered in section \ref{ssec:O4} by introducing a rescaling of $C(u)$ that depends on $u$. This scaling is chosen so that the amplitude of samples is $\mathcal{O}(1)$ with respect to $u$. The rescaled family can be shown to satisfy Assumptions \ref{ass:op}, and a minorization condition as in Lemma \ref{lem:minor_op} can also be shown to hold when the state space is finite-dimensional. From this we can deduce that 
the resulting discretized process will still be ergodic.

Assume $D\subseteq \R^d$ and define the negative Laplacian $-\Delta$ on $D(-\Delta)$,
\[
D(-\Delta) = \left\{u \in H^2(D;\R)\,\bigg|\,\frac{\dee u}{\dee \nu}(x) = 0\text{ for }x \in \partial D\right\},
\]
where $\nu$ is the outward normal to $\partial D$. Given $\alpha > d/2$, $\sigma > 0$, we define $P = -\Delta$ and 
\begin{align}
\label{eq:cov_numerics}
C(u)^{-1} = \sigma^{-2}(P + \Gamma(u))^{\alpha/2} \Gamma(u)^{d/2-\alpha}(P + \Gamma(u))^{\alpha/2}
\end{align}
where $\big(\Gamma(u)v\big)(x) = F\big(u(x)\big)v(x)$. The scaling introduced is inspired by the SPDE representation of Whittle-Mat\'ern distributions \citet{matern_spde}; if $F(u) = \tau^2$ is chosen to be constant, then modulo boundary conditions, samples from a centred Gaussian distribution with covariance $C(u)$ are samples from a Whittle-Mat\'ern distribution. In particular, $\tau$ corresponds to the inverse length-scale of samples, and samples almost-surely have $s$ Sobolev and H\"older and derivatives for any $s < \alpha-d/2$.

For numerical experiments, we take
\[
F(u) = \min\{F_- + ae^{bu^2},F_+\}
\]
for some $F_+,F_-,a,b > 0$. In particular, in one spatial dimension we take {$F_+ = 150^2$, $F_- = 200$,  $a=100$ and $b = 2$}. In two dimensions, we take $F_+ = 150^2$, $F_- = 50$, $a=25$ and $b=0.3$. We take $\alpha = 4$ in both cases, and choose $\sigma$ such that $\mathbb{E}\big(u(x)^2\big) \approx 1$. These parameter choices were made empirically to ensure interesting structure of the samples. In order to generate samples at a given level, the negative Laplacian $P$ is constructed using a finite-difference method. Given $u$, the operator $A(u)$ is then computed, 
\[
A(u) := \sigma^{-1}\Gamma(u)^{d/4-\alpha/2}(P+\Gamma(u))^{\alpha/2},
\]
so that $v \sim N(0,C(u))$ solves the SPDE $A(u)v = \xi$, where $\xi$ is white noise.

In Figure \ref{f:prior_1d_op} we show samples of the deep Gaussian process on domain $D=(0,1)$, sampled on the uniform grid $x_i = \frac{i-1}{1000}$, for $i=1,\ldots,1001$. We show 4 independent realizations of the first seven layers of the process  -- each row corresponds to a given layer $u_n$. The anisotropy of the length-scale is evident in levels beyond $u_0$, and the effect of ergodicity is evident, with deeper levels having similar properties. Compared to the covariance function construction, local effects are less prominent, though a greater level of anisotropy could potentially be obtained by making an alternative choice of $F(\cdot)$. Figure \ref{f:prior_2d_op} shows the same experiments on domain $D=(0,1)^2$, sampled on the tensor product of the one-dimensional points $x_i^1 = \frac{i-1}{150}$, for $i=1,\ldots,151$, and the same effects are observed. {Figure \ref{f:norm_trace} shows the trace of the norm of a DGP $\{u_n\}$ with $d=1$, along with the running mean of these norms; the rapid convergence of the mean reflects the ergodicity of the chain.}

\begin{figure}
\begin{center}
\begin{overpic}[width=\linewidth,trim=1cm 1.1cm 2.3cm 1.5cm,clip]{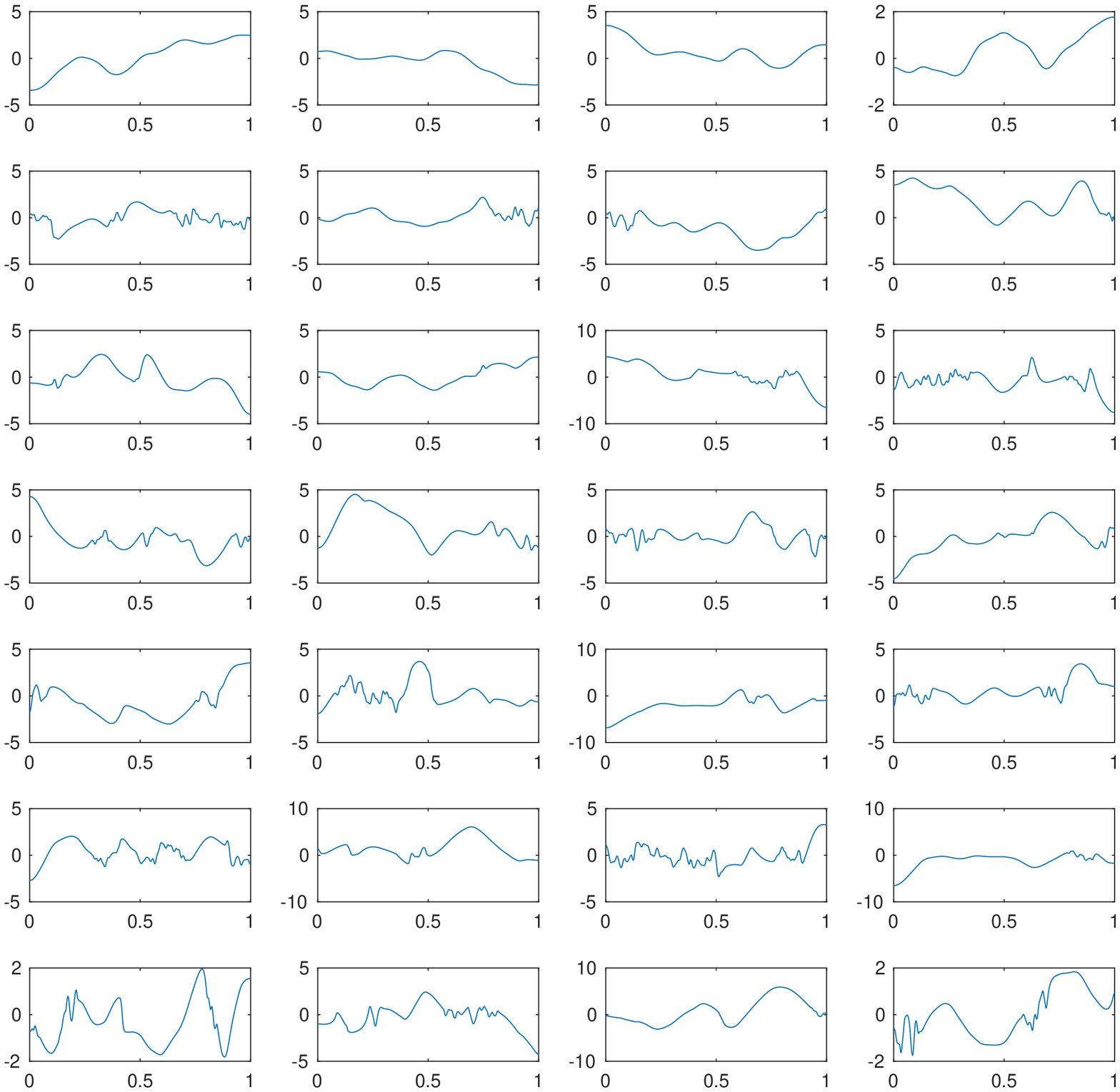}
\put(1.5,90){\rotatebox{90}{$u_0$}}
\put(1.5,76.5){\rotatebox{90}{$u_1$}}
\put(1.5,63.0){\rotatebox{90}{$u_2$}}
\put(1.5,49.5){\rotatebox{90}{$u_3$}}
\put(1.5,36.0){\rotatebox{90}{$u_4$}}
\put(1.5,22.5){\rotatebox{90}{$u_5$}}
\put(1.5,9.0){\rotatebox{90}{$u_6$}}

\put(11.5,0){\small Sample 1}
\put(35.8,0){\small Sample 2}
\put(60.1,0){\small Sample 3}
\put(84.5,0){\small Sample 4}
\end{overpic}
\end{center}
\caption{Four independent realizations of the first seven layers of a deep Gaussian process, in one spatial dimension, using the covariance operator construction described in subsection \ref{ssec:cov_op}. Each column corresponds to an independent chain, and layers $u_0,u_1,\ldots,u_6$ are shown from top-to-bottom.}
\label{f:prior_1d_op}
\end{figure}

\begin{figure}
\begin{center}
\begin{overpic}[width=\linewidth,trim=1.0cm 1.3cm 3.2cm 1.8cm,clip]{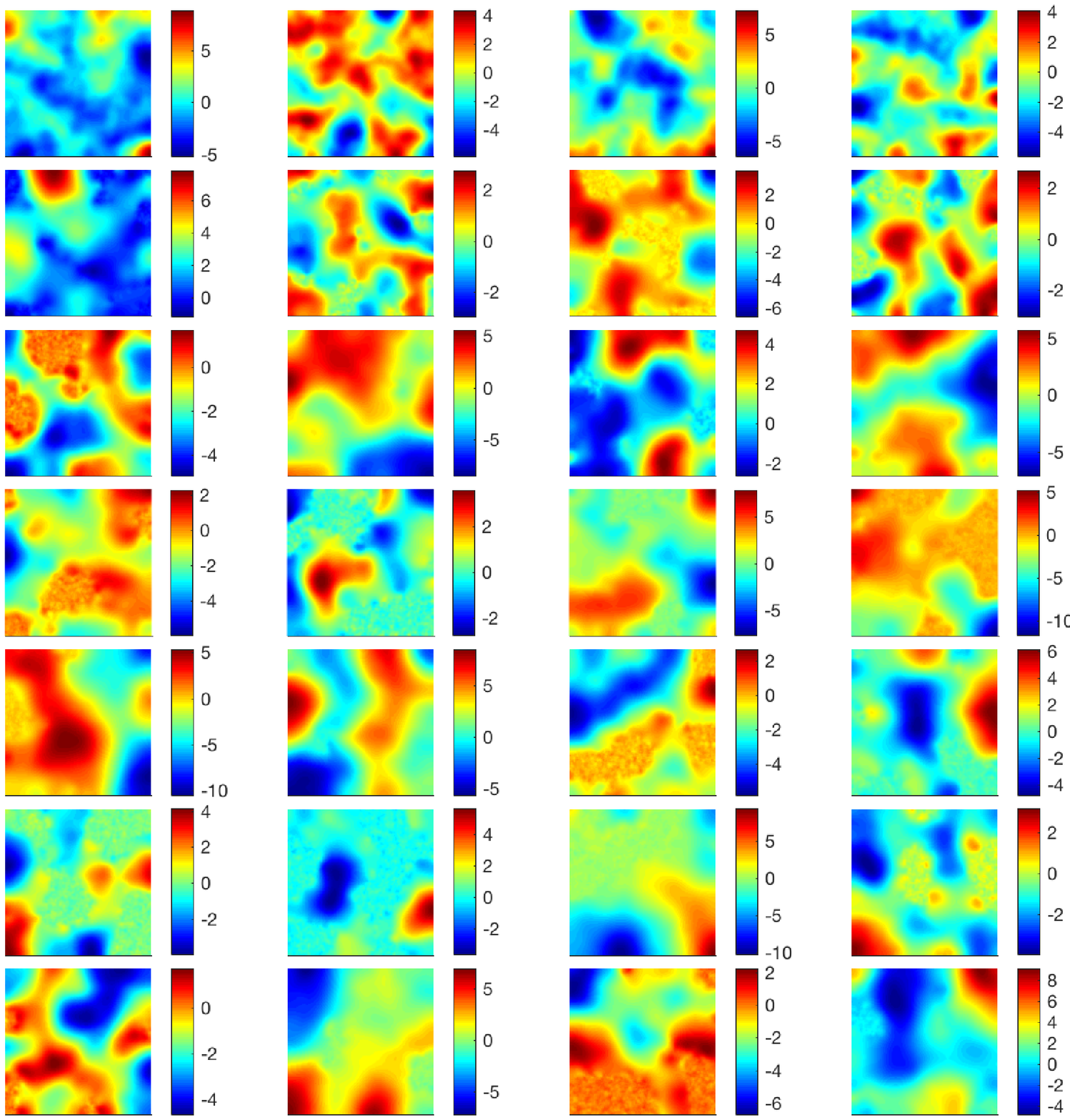}
\put(3,89.5){\rotatebox{90}{$u_0$}}
\put(3,76.3){\rotatebox{90}{$u_1$}}
\put(3,63.1){\rotatebox{90}{$u_2$}}
\put(3,49.9){\rotatebox{90}{$u_3$}}
\put(3,36.7){\rotatebox{90}{$u_4$}}
\put(3,23.5){\rotatebox{90}{$u_5$}}
\put(3,10.3){\rotatebox{90}{$u_6$}}

\put(11,0){\small Sample 1}
\put(34.3,0){\small Sample 2}
\put(57.6,0){\small Sample 3}
\put(80.9,0){\small Sample 4}
\end{overpic}
\end{center}
\caption{Four independent realizations of the first seven layers of a deep Gaussian process, in two spatial dimensions, using the covariance operator construction described in subsection \ref{ssec:cov_op}. Each column corresponds to an independent chain, and layers $u_0,u_1,\ldots,u_6$ are shown from top-to-bottom.}
\label{f:prior_2d_op}
\end{figure}

\begin{figure}
\begin{center}
\includegraphics[width=0.9\linewidth,trim=1.5cm 0cm 1.5cm 0cm,clip]{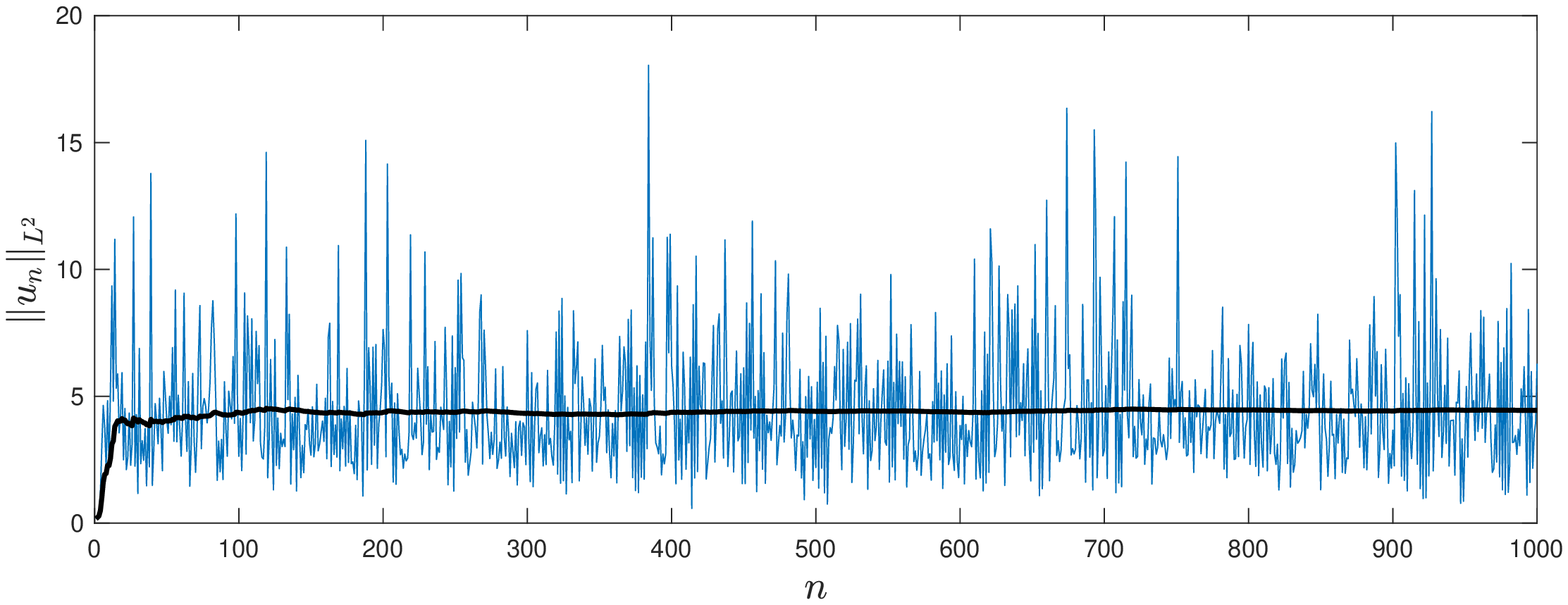}
\end{center}
\caption{{The trace of the norm of $u_n$ versus $n$ for a 1000 layer DGP $\{u_n\}$ as in Figure \ref{f:prior_1d_op}. The thick black curve shows the running mean of the norms.}}
\label{f:norm_trace}
\end{figure}

We emphasize that our perspective
on inference includes quite general inverse problems, and is not limited
to the problems of regression and classification which dominate much of
classical machine learning; this broad perspective on the potential
for the methodology affects the choice of algorithms that we study as we
do not exploit any of the special structures that arise in regression
and classification.

The deep Gaussian processes discussed in the previous sections were introduced with the idea of providing flexible prior distributions for inference, for example in inverse problems. The structure of such problems is as follows. We have data $y \in \R^J$ arising via the model
\begin{align}
\label{eq:data_c}
y = \mathcal{G}(u) + \eta
\end{align}
where $\eta$ is a realization of some additive noise, and $\mathcal{G}:X\to\R^J$ is a (typically non-linear) forward map. The map $\mathcal{G}$ may involve, for example, solution of a
partial differential equation which takes function $u$ as input, or point evaluations of a 
function $u$, regression. In this paper we will fix $X = \cH^{N}$, writing $u = (u_0,\ldots,u_{N-1}) \in X$; our prior beliefs on $u$ will then be characterized by the first $N$ states of a Markov chain of a form considered in the previous sections. Note that the map $\mathcal{G}$ could incorporate a projection map if the dependence is only upon a single state $u_{N-1}$; indeed this is the canonical example -- the variables $(u_0,\dots,u_{N-2})$ are 
viewed as hyperparameters in a prior on the parameter $u_{N-1}.$

\subsubsection{Algorithms}
We now turn to the design of algorithms for the Bayesian inference problems of
sampling $u|y.$ {As already mentioned above, we are typically only interested in sampling the deepest layer $u_{N-1}|y$. However, due to the hierarchical definition of $u_{N-1}$ given all the components of $u$, our algorithms work with the full set of layers $u$. }Since the components of $u$ are functions, and hence infinite
dimensional objects in general, a guiding principle is to design algorithms
which are well-defined on function space, an approach to MCMC inference reviewed 
in \citet{cotter2013mcmc}; the value of this approach is that it leads to
algorithms whose mixing time is not dependent on the number of mesh points used 
to represent the function to be inferred.  For simplicity of exposition we assume that the
observational noise $\eta$ is distributed as $N(0,\Gamma)$; this is not
central to our developments but makes the exposition concrete. 
Recalling that the Markov chain defining the prior beliefs is given by (\ref{eq:cg3}), we can consider the unknowns in the problem to be the variables $u = (u_0,\ldots,u_{N-1})$, which are correlated under the prior, or the variables $\xi = (\xi_0,\ldots,\xi_{N-1})$, where
we define $\xi_0=u_0$, which are independent under the prior. These variables are related via $u = T(\xi)$, where the components of the deterministic map $T:X\to X$ are defined iteratively by
\begin{align*}
T_1(\xi_0,\ldots,\xi_{N-1}) &= \xi_0,\\
T_{n+1}(\xi_0,\ldots,\xi_{N-1}) &= L(T_n(\xi_{0},\ldots,\xi_{N-1}))\xi_n,\;\;\;n = 1,\ldots,N-1.
\end{align*}
The data may then be expressed in terms of $\xi$ rather than $u$:
\begin{align}
\label{eq:data_nc}
y = \tilde{\mathcal{G}}(\xi) + \eta = \mathcal{G}(T(\xi)) + \eta
\end{align}
where our prior belief on $\xi$ is that its components are i.i.d. Gaussians. 
To be consistent with the notation introduced in \citet{papaspiliopoulos2007general,yu2011center}  (\ref{eq:data_c}) will be referred to as the \emph{centred} model and (\ref{eq:data_nc}) will be referred to as the \emph{non-centred} 
model. The space $\cH$ may be chosen differently in the
centred and non-centred cases.

Associated with the two data models are two likelihoods: $\bbP(y|u)$ and $\bbP(y|\xi)$. Assuming that the observational noise $\eta\sim N(0,\Gamma)$ is Gaussian, where $\Gamma \in \R^{J\times J}$ is a positive definite covariance matrix, the likelihoods are given by
\begin{align*}
\bbP(y|u) = \frac{1}{Z(y)}\exp\big(-\Phi(u;y)\big),\quad &\Phi(u;y):= \frac{1}{2}\big|\Gamma^{-\frac{1}{2}}(y-\mathcal{G}(u))\big|^2,\\
\bbP(y|\xi) = \frac{1}{\tilde{Z}(y)}\exp\big(-\tilde{\Phi}(\xi;y)\big),\quad &\tilde{\Phi}(\xi;y):= \frac{1}{2}\big|\Gamma^{-\frac{1}{2}}(y-\tilde{\mathcal{G}}(\xi))\big|^2.
\end{align*}
We may then apply Bayes' theorem to write down the posterior distributions $\bbP(u|y)$ and $\bbP(\xi|y)$:
\begin{align*}
\bbP(u|y)\propto \bbP(y|u)\bbP(u) \propto \exp\big(-\Phi(u;y)\big)\bbP(u),\\
\bbP(\xi|y)\propto \bbP(y|\xi)\bbP(\xi) \propto \exp\big(-\tilde{\Phi}(\xi;y)\big)\bbP(\xi).
\end{align*}

We know from \citet{cotter2013mcmc} that it is straightforward to design
algorithms to sample $\bbP(\xi|y)$ which are well-defined in infinite
dimensions, exploiting the fact that $\bbP(\xi)$ is Gaussian. 
An example of such an algorithm is:

\begin{algorithm}[Non-Centred Algorithm]{\textcolor{white}.}\\\vspace{-0.4cm}
\begin{enumerate}
\item  Fix $\beta_0,\ldots,\beta_{N-1} \in (0,1]$ and define $B = \mathrm{diag}(\beta_j)$. Choose initial state $\xi^{(0)} \in X$, and set $u^{(0)} = T(\xi^{(0)}) \in X$. Set $k = 0$.
\item Propose $\hat{\xi}^{(k)} = (I - B^2)^{\frac{1}{2}} \xi^{(k)} + B \zeta_j^{(k)},\quad\zeta^{(k)}\sim N(0,I)$.
\item Set $\xi^{(k+1)} = \hat{\xi}^{(k)}$ with probability
\fontsize{9}{12}\selectfont
\[
\alpha_{k} = \min\left\{1,\exp\left(\Phi\big(T(\xi^{(k)});y\big) - \Phi\big(T(\hat{\xi}^{(k)});y\big)\right)\right\};
\]
\fontsize{10}{12}\selectfont
otherwise set $\xi^{(k+1)} = \xi^{(k)}$.

\item Set $k \mapsto k+1$ and go to 1.
\end{enumerate}
\end{algorithm}

This algorithm produces a chain $\{\xi^{(k)}\}_{k\in\N}$ that samples $\bbP(\xi|y)$ 
in stationarity; and $\{T(\xi^{(k)})\}_{k\in\N}$ will be samples of $\bbP(u|y).$ 
By working in non-centred coordinates we have been able to design this 
algorithm which is well-defined on function space. 
d{If we were to work with the centred coordinates $u$ directly, the algorithm would not be well-defined on function space: in infinite dimensions, each family of measures $\{\mathbb{P}(u_n|u_{n-1})\}_{u_{n-1} \in \cH}$ will typically be mutually singular, and so a proposed update $u\mapsto \hat{u}$ will almost surely be rejected. To see why this rejection occurs in practice, in high finite dimensions $K$, notice that the acceptance probability for an update $u\mapsto \hat{u}$ will involve the ratios of the Gaussian densities $N(u_n;0,C(u_{n-1}))$ and $N(u_n;0,C(\hat{u}_{n-1}))$. These densities will decay to zero as the dimension $K$ is increased, and their ratio will only be well-defined in the limit if the measures are equivalent; consequently, the Markov chain will mix very poorly. Working with the non-centred coordinates $\xi$, the prior does not appear in the acceptance probability and so this issue is circumvented. Another advantage of using the non-centred coordinates is that there is no need to calculate the (divergent) log determinants which appear in the centred acceptance probability, avoiding potential numerical issues. These issues are discussed 
in greater depth and generality in \citet{omiros}. For the reasons set-out 
in that paper, including those above, we have used only the non-centred algorithm in what follows.}
When the forward model $\mathcal{G}(u) = Au$ is linear, the non-centred algorithm can be combined with standard Gaussian process regression techniques via the identity
\[
\PP(\dee u_N|y) = \int_X \PP(\dee u_N|u_{N-1},y)\PP(\dee u_{N-1}|y).
\]
The distribution $\PP(\dee u_N|u_{N-1},y) = N(m_y(u_{N-1}),C_y(u_{N-1}))$ is Gaussian, where expressions for $m_y,C_y$ are known, and so direct sampling methods are available. On the other hand, we have that $\PP(y|u_{N-1}) = N(0,AC(u_{N-1})A^* + \Gamma)$, and so we may use the non-centred algorithm to robustly sample the measure 
{\begin{align*}
\PP(\dee u_{N-1}|y) &= \exp(-\Psi(u_{N-1};y))\PP(\dee u_{N-1}),\\
\Psi(u_{N-1};y) &= \frac{1}{2}\|y\|_{AC(u_{N-1})A^* + \Gamma}^2 + \frac{1}{2}\log\det(AC(u_{N-1})A^* + \Gamma),
\end{align*}}
after reparametrizing in terms of $\xi$. This approach can be viable even when the data is particularly informative so that $\Phi$ is very singular -- this singularity does not in general pass to $\Psi$. It is this approach that we use for the simulations in the following subsections. An alternative approach not based on MCMC would be to use the non-centred parameterization of the 
Ensemble Kalman Filter \citet{CIRS17} which we have successfully
implemented in the context of the deep Gaussian processes of this paper,
but do not show here for reasons of brevity. 

\subsection{Application to Regression}

\subsubsection{One-Dimensional Simulations}

We consider first the case $D = (0,1)$, where the forward map is given by a number of point evaluations: $\mathcal{G}_j(u) = u(x_j)$ for some sequence $\{x_j\}_{j=1}^J\subseteq D$. We compare the quality of reconstruction versus both the number of point evaluations and the number of levels in the deep Gaussian prior. We use the same parameters for the family of covariance operators as in subsection \ref{ssec:cov_op}. The base layer $u_0$ is taken to be Gaussian with covariance of the form (\ref{eq:cov_numerics}), with $\Gamma(u)\equiv 20^2$.

The true unknown field $u^\dagger$ is given by the indicator function $u^\dagger = \mathds{1}_{(0.3,0.7)}$, shown in Figure \ref{f:truth_1d}. It is generated on a mesh of $400$ points, and three data sets are created wherein it is observed on uniform grids of $J=25$, $50$ and $100$ points, and corrupted by white noise with standard deviation $\gamma = 0.02$. Sampling is performed on a mesh of 200 points to avoid an inverse crime \citet{kaipio2006statistical}. $10^6$ samples are generated per chain, with the first $2\times 10^5$ discarded as burn-in when calculating means. The jump parameters $\beta_j$ are adaptively tuned to keep acceptance rates close to $30\%$.

In these experiments the deepest field is labelled as $u_N$, rather than
as $u_{N-1}$ as in the statement of the algorithm; this is purely for
notational convenience, of course.
In Figure \ref{f:means_1d} the means of the deepest field $u_N$ and of the length-scales associated with each hidden layer are shown, that is, approximations to $\mathbb{E}\big(u_N\big)$ and $\mathbb{E}\big(F(u_j)^{\frac{1}{2}}\big)$ for each $j=0,\ldots,N-1$. We see that, in all cases, the reconstructions of $u^\dagger$ are visually similar when two or more layers are used, and similar length-scale fields $\mathbb{E}\big(F(u_{N-1})^{\frac{1}{2}}\big)$ are obtained in these cases. The sharpness of these length-scale fields is related to the amount of data. Additionally, when $N=4$ and $J=100$ the location of the discontinuities is visible in the estimate for $\mathbb{E}\big(F(u_{N-2})^{\frac{1}{2}}\big)$, suggesting the higher quality data can influence the process more deeply. When $J=50$ or $J=25$, this layer does not appear to be significantly informed. When a single layer prior is used, the reconstruction fails to accurately capture the discontinuities. {Figure \ref{f:means_1d} also shows bands of quantiles of the values $u(x)$ under the posterior, illustrating their distribution; in particular the lack of symmetry and disagreement of the means and medians show that the posterior is clearly non-Gaussian. Uncertainty increases both as the number of observations $J$ and the layer $n$ in the chain is increased. Note in particular the over-confidence of the shallow Gaussian process posterior: the truth is not contained within $95\%$ credible intervals in all cases.}

In Table \ref{tab:errs_1d} we show the $L^1$-errors between the true field and the posterior means arising from the different setups. The errors decrease as the number of observation points is increased, as would be expected. Additionally, when $J=100$ and $J=50$, the accuracy of the reconstruction increases with the number of layers, though the most significant increase occurs when increasing from 1 to 2 layers. When $J=25$, the error increases beyond 2 layers, suggesting that some balance is required between the quality of the data and the flexibility of the prior.

In Figure \ref{f:means_1d_1M} we replace the uniformly spaced observations with $10^6$ randomly placed observations, to illustrate the effect of very high quality data. With 3 or 4 layers, more anisotropic behavior is observed in the length-scale field. Additionally, the layer $u_{N-2}$ is much more strongly informed than the cases with fewer observations, though the layer $u_{N-3}$ in the case $N=4$ does not appear to be informed at all, indicating a limitation on how deeply the process can be influenced by data. The corresponding errors are shown in Table \ref{tab:errs_1d} -- as in the cases $N=50,100$, more layers increases the accuracy of the mean, with diminishing returns for each additional layer. Note that higher accuracy could be attained in the single layer case by adjusting the constant length-scale parameter.

{Finally, in Figure \ref{f:means_1d_h}, we consider the same experiment as in Figure \ref{f:means_1d}, except observations are limited to the subset $(0,0.5)$ of the domain. Uncertainty is naturally higher in the unobserved portion of the domain. Uncertainty also increases in the observed layer $u_N$ as $N$ is increased; this could suggest that deep Gaussian processes may provide better generalization to unseen data than shallow Gaussian processes -- note that the truth has much higher probability under the posterior with 4 layers versus just 1.}

\begin{figure}
\begin{center}
\includegraphics[width=0.6\linewidth,trim=1.5cm 0cm 1.5cm 0cm,clip]{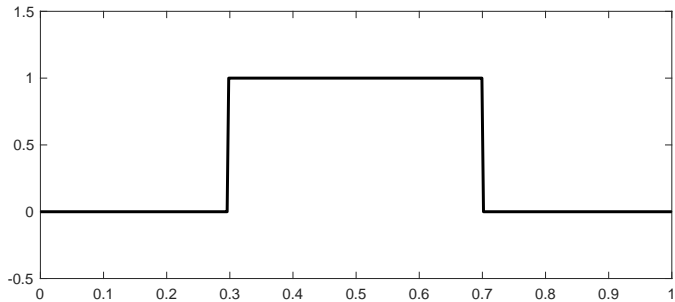}
\end{center}
\caption{The true field used to generate the data for the one-dimensional inverse problem.}
\label{f:truth_1d}
\end{figure}

\begin{figure}
\begin{center}
\includegraphics[width=\linewidth,trim=2.5cm 1.5cm 3cm 0cm,clip]{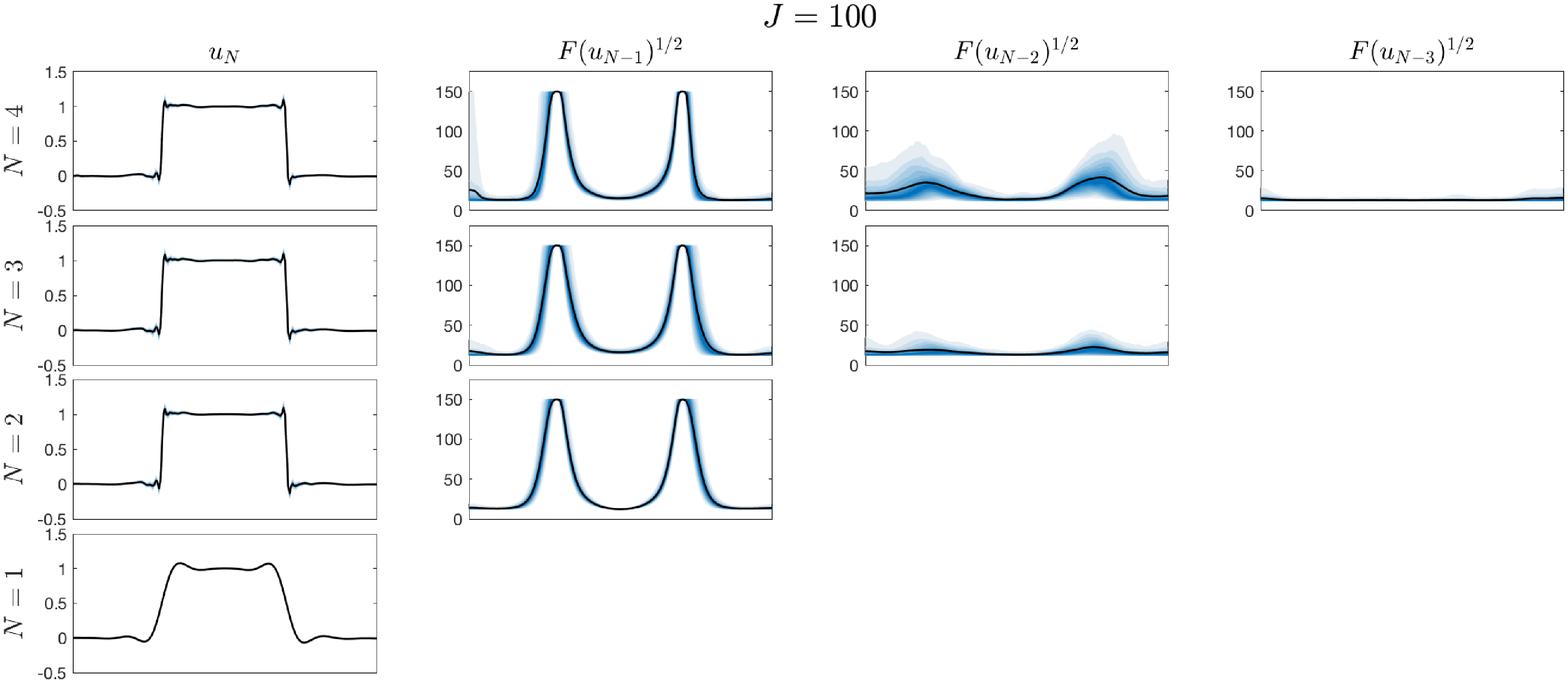}\\
\vspace{0.5cm}
\includegraphics[width=\linewidth,trim=2.5cm 1.5cm 3cm 0cm,clip]{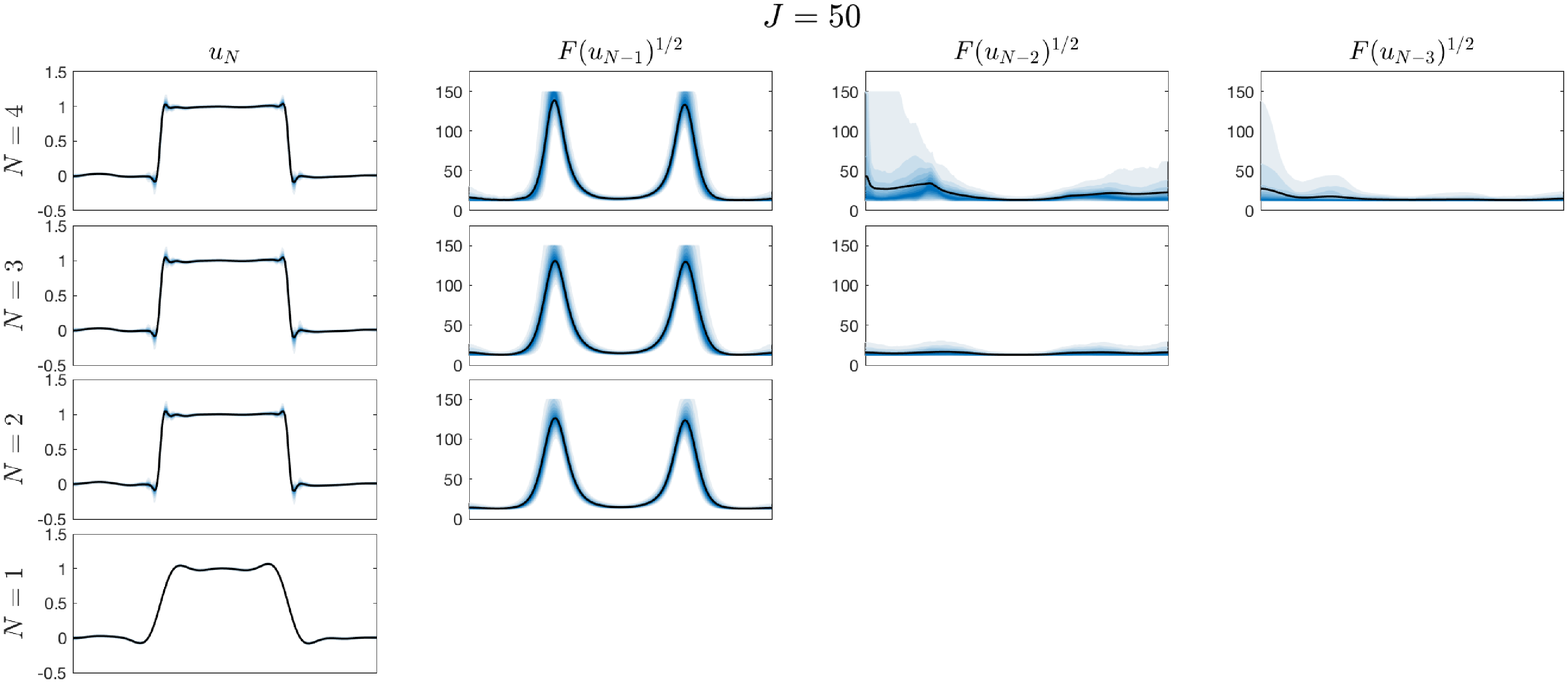}\\
\vspace{0.5cm}
\includegraphics[width=\linewidth,trim=2.5cm 1.5cm 3cm 0cm,clip]{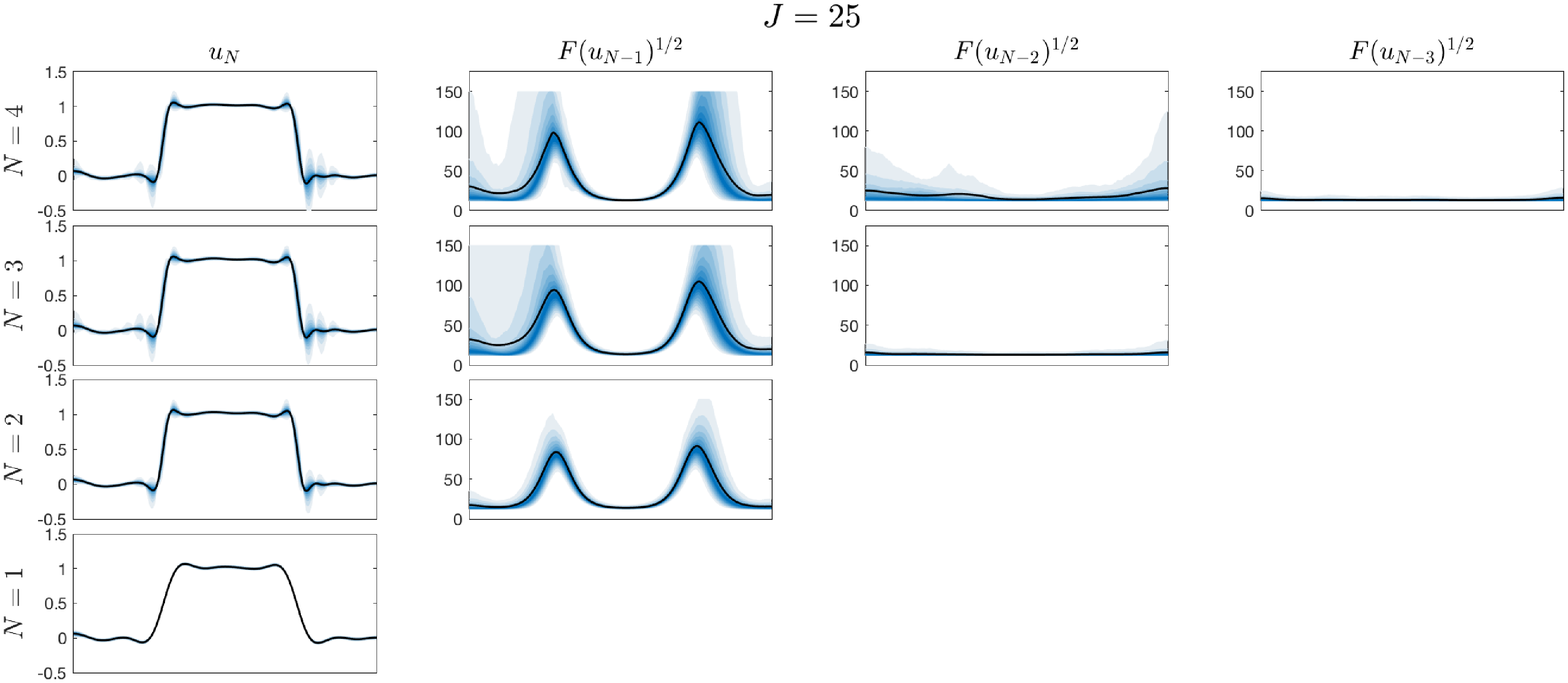}
\end{center}
\caption{{Estimates of posterior means (solid curves) and $5-95\%$ quantiles (shaded regions) arising from one-dimensional inverse problem. Number of data points taken are $J=100$ (top block), $J=50$ (middle block), $J=25$ (bottom block). From left-to right, results for $u_N$, $F(u_{N-1})^{\frac{1}{2}}$, $F(u_{N-2})^{\frac{1}{2}}$, $F(u_{N-3})^{\frac{1}{2}}$ are shown. From top-to-bottom within each block, $N=4,3,2,1$.}}
\label{f:means_1d}
\end{figure}

\begin{table}
\caption{The $L^1$-errors $\|u^\dagger - \mathbb{E}(u_N)\|_{L^1}$ between the true field and sample means for the one-dimensional simulations shown in Figure \ref{f:means_1d}{, for different numbers of data points $J$ and layers $N$}. Also shown are the corresponding errors for the simulations shown in Figure \ref{f:means_1d_1M}}.
\label{tab:errs_1d}
\bgroup
\def\arraystretch{1.5}
\begin{tabular}{c|r r r r}
\hline
$J$ & 1 layer & 2 layers & 3 layers & 4 layers\\ \hline
$100$ & 0.0485 & 0.0200 & 0.0198 & {\bf 0.0196}\\
$50$ & 0.0568 & 0.0339 & 0.0339 & {\bf 0.0337}\\
$25$ & 0.0746 & {\bf 0.0658} & 0.0667 & 0.0670\\
\hline
$10^6$ & 0.0131 & 0.000145 & {\bf 0.000133} & {\bf 0.000133}\\
\hline
\end{tabular}
\egroup
\end{table}

\begin{figure}
\begin{center}
\includegraphics[width=\linewidth,trim=2.5cm 1.5cm 3cm 0cm,clip]{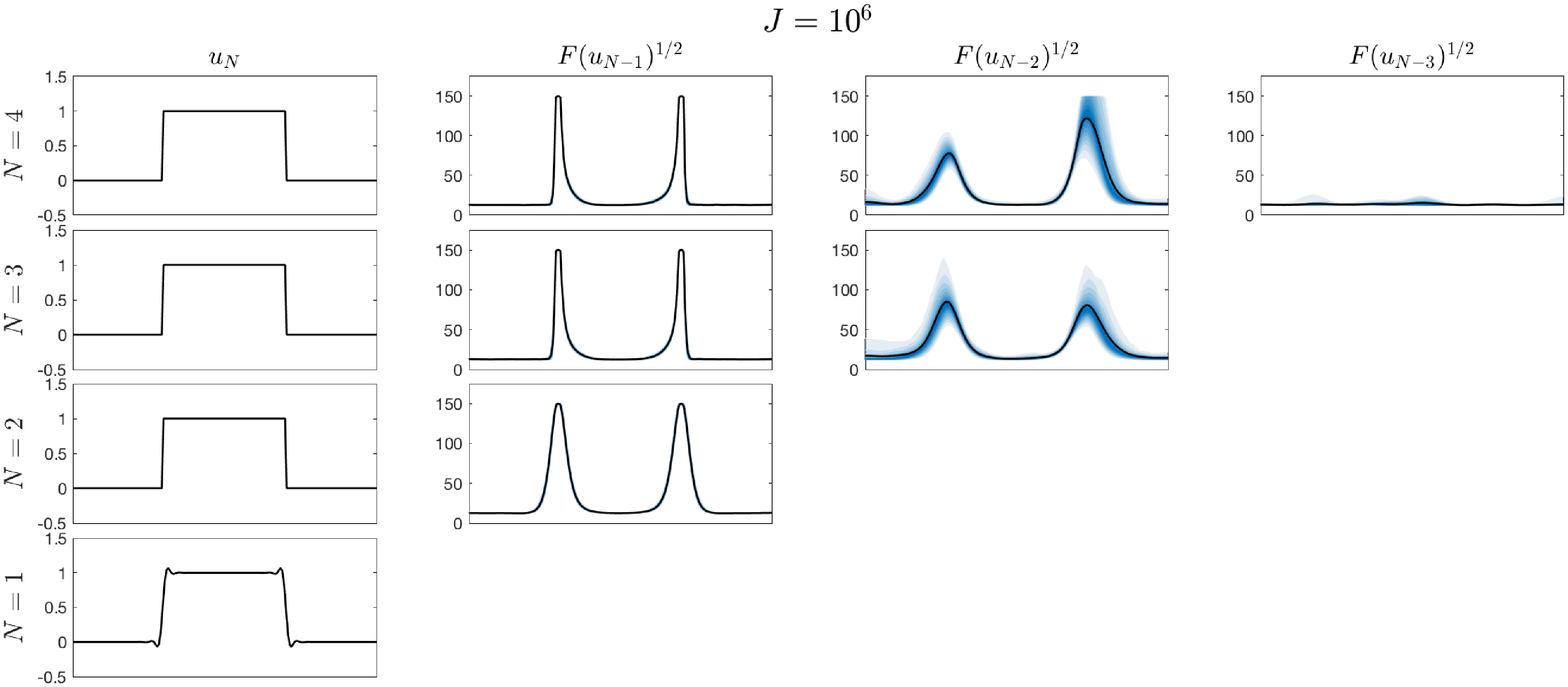}
\end{center}
\caption{{Estimates of posterior means (solid curves) and $5-95\%$ quantiles (shaded regions) arising from one-dimensional inverse problem, with $J=10^6$ {data points} Number of data points taken are $J=100$ (top block), $J=50$ (middle block), $J=25$ (bottom block). From left-to right, results for $u_N$, $F(u_{N-1})^{\frac{1}{2}}$, $F(u_{N-2})^{\frac{1}{2}}$, $F(u_{N-3})^{\frac{1}{2}}$ are shown. From top-to-bottom within each block, $N=4,3,2,1$.}}

\label{f:means_1d_1M}
\end{figure}

\begin{figure}
\begin{center}
\includegraphics[width=\linewidth,trim=2.5cm 1.5cm 3cm 0cm,clip]{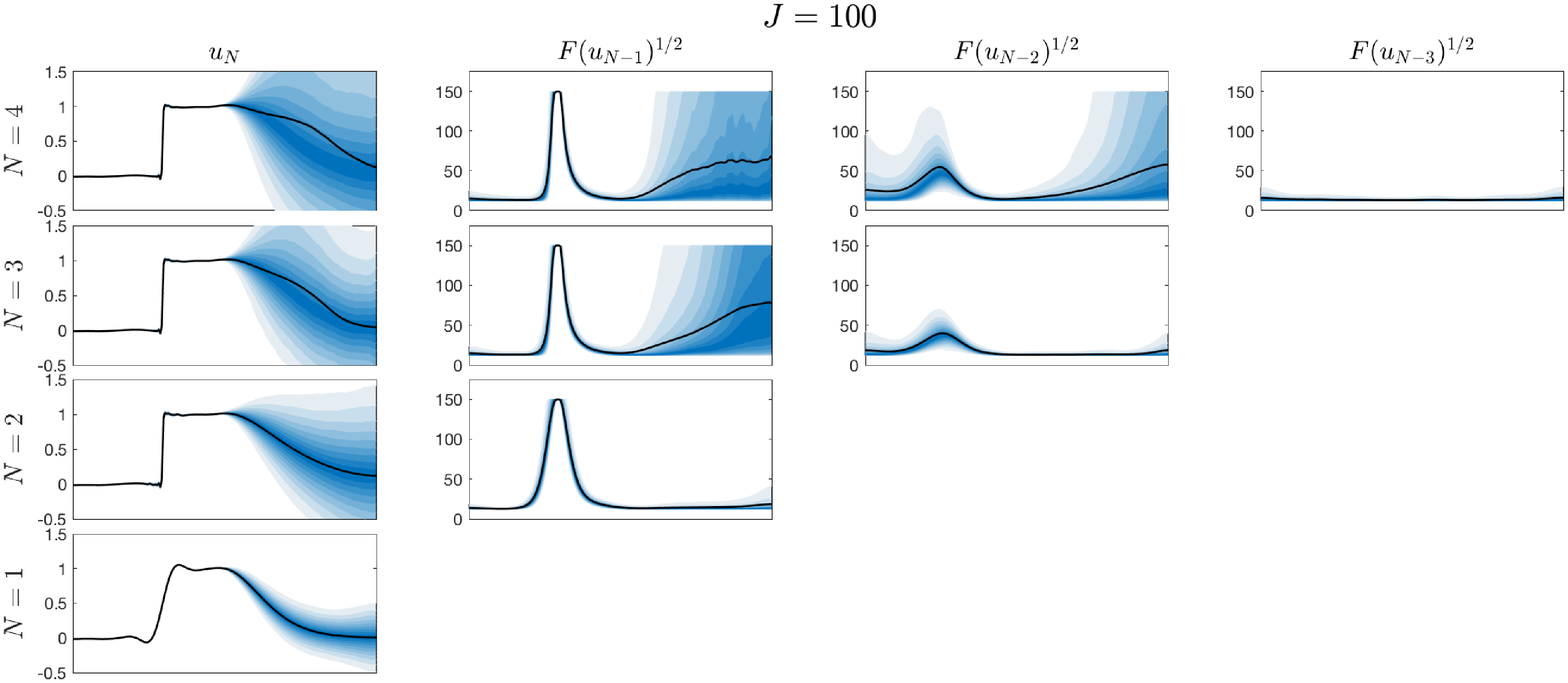}\\
\vspace{0.5cm}
\includegraphics[width=\linewidth,trim=2.5cm 1.5cm 3cm 0cm,clip]{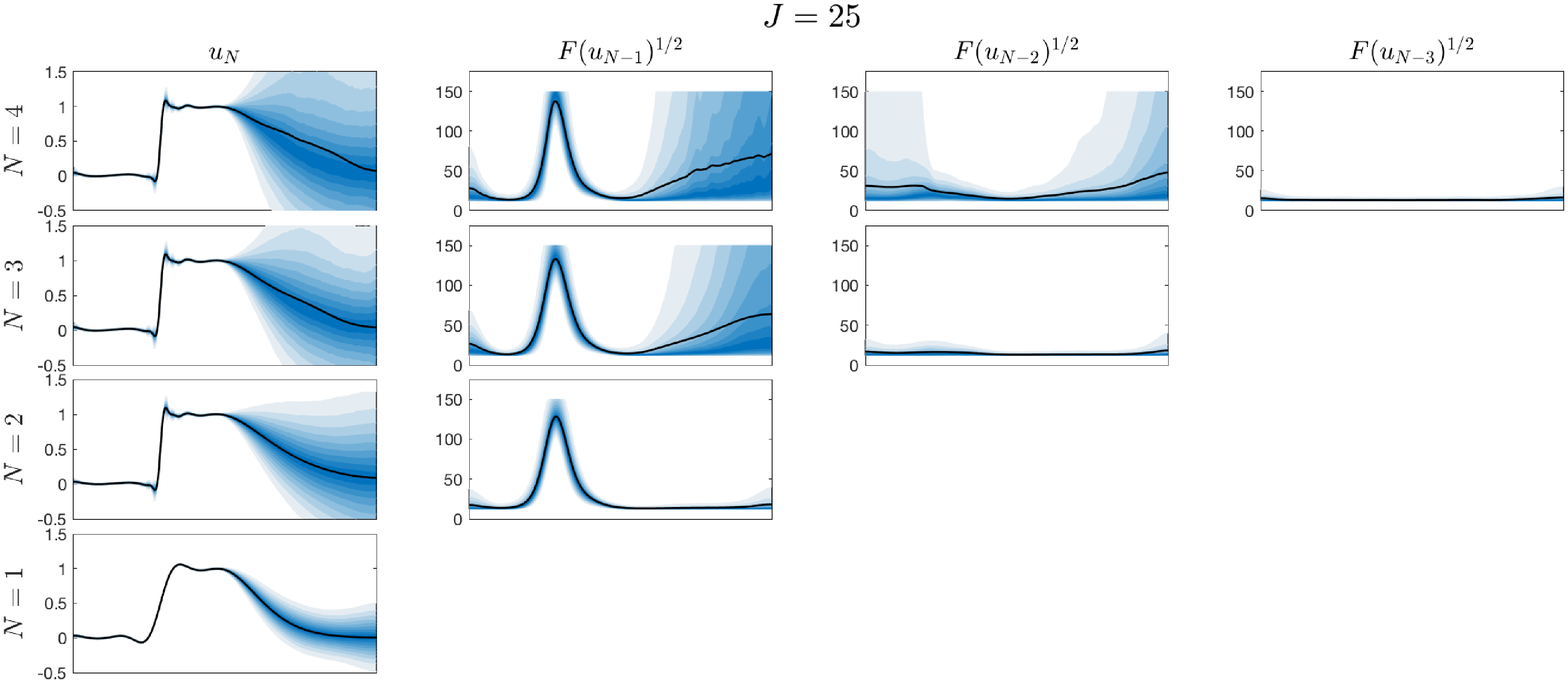}
\end{center}
\caption{{Estimates of posterior means (solid curves) and $5-95\%$ quantiles (shaded regions) arising from one-dimensional inverse problem. Number of data points taken are $J=100$ (top block), $J=50$ (middle block), $J=25$ (bottom block). From left-to right, results for $u_N$, $F(u_{N-1})^{\frac{1}{2}}$, $F(u_{N-2})^{\frac{1}{2}}$, $F(u_{N-3})^{\frac{1}{2}}$ are shown. From top-to-bottom within each block, $N=4,3,2,1$.}}
\label{f:means_1d_h}
\end{figure}

\subsubsection{Two-Dimensional Simulations}

We now consider the case $D = (0,1)^2$, again where the forward map is given by a number of point evaluations. We fix the number of point observations $J=2^{10}$, on a $2^5\times 2^5$ uniform grid. We again compare quality of reconstruction versus the number of point evaluations and the number of levels in the deep Gaussian prior, and use the same parameters for the family of covariance operators as in subsection \ref{ssec:cov_op}. The base layer $u_0$ is taken to be Gaussian with covariance of the form (\ref{eq:cov_numerics}), with $\Gamma(u)\equiv 20^2$.

The true unknown field $u^\dagger$ is constructed as a linear combination of truncated trigonometric functions with different length-scales, and shown in Figure \ref{f:truth_2d} along with its contours. It is given by
\begin{align*}
u^\dagger(x,y) &= \cos(2\pi x)\cos(2\pi y) + \sin(4\pi x)\sin(4\pi y)\mathds{1}_{(1/4,3/4)^2}(x,y)\\
&\hspace{1cm}+\sin(8\pi x)\sin(8\pi y)\mathds{1}_{(1/2,3/4)^2}(x,y)\\
&\hspace{1cm}+\sin(16\pi x)\sin(16\pi y)\mathds{1}_{(1/4,1/2)^2}(x,y).
\end{align*}
It is generated on a uniform square mesh of $2^{14}$ points, and two data sets are created wherein it is observed on uniform square grid of $J=2^{10}, 2^8$ points, and corrupted by white noise with standard deviation $\gamma=0.02$. Sampling is performed on a mesh of $2^{12}$ points to again avoid an inverse crime. $4\times 10^5$ samples are generated per chain, with the first $2\times 10^5$ discarded as burn-in when calculating means. Again the jump parameters $\beta_j$ are adaptively tuned to keep acceptance rates close to $30\%$.

In Figure \ref{f:means_2d}, analogously to Figure \ref{f:means_1d}, the means of $u_N$ and of the length-scales associated with each layer are shown, for $N=1,2,3$. When $J=2^{10}$, reconstructions are similar, though quality is generally proportional to the number of layers. In particular the, effect of too short a length-scale is evident in the case $N=1$, in the regions where the length-scale should be larger, and conversely the effect of too long a length-scale is evident in the cases $N=1,2$ in the region where the length-scale should be the shortest. In the cases $N=2,3$, the length-scale fields $\mathbb{E}\big(F(u_{N-1})^{\frac{1}{2}}\big)$ are similar, though in the case $N=3$ more accurately captures the true length-scales. When $J=2^8$ the reconstructions are again similar, though there is now less accuracy in the shapes of the contours. In particular, the effect of too short a length-scale is especially evident in the case $N=1$. The values of the reconstructed fields in the area of shortest length-scale are inaccurate in all cases -- the positions of the observation points meant that the actual values of the peaks were not reflected in the data. The fields $\mathbb{E}\big(F(u_{N-1})^{\frac{1}{2}}\big)$ have similar structure to the case $J=2^{10}$, though less accurately represent the true length scales. The $L^2$-errors between the means and the truth are shown in Table \ref{tab:errs_2d}

\begin{figure}
\begin{center}
\includegraphics[width=0.45\linewidth,trim=1cm 0cm 1cm 0cm,clip]{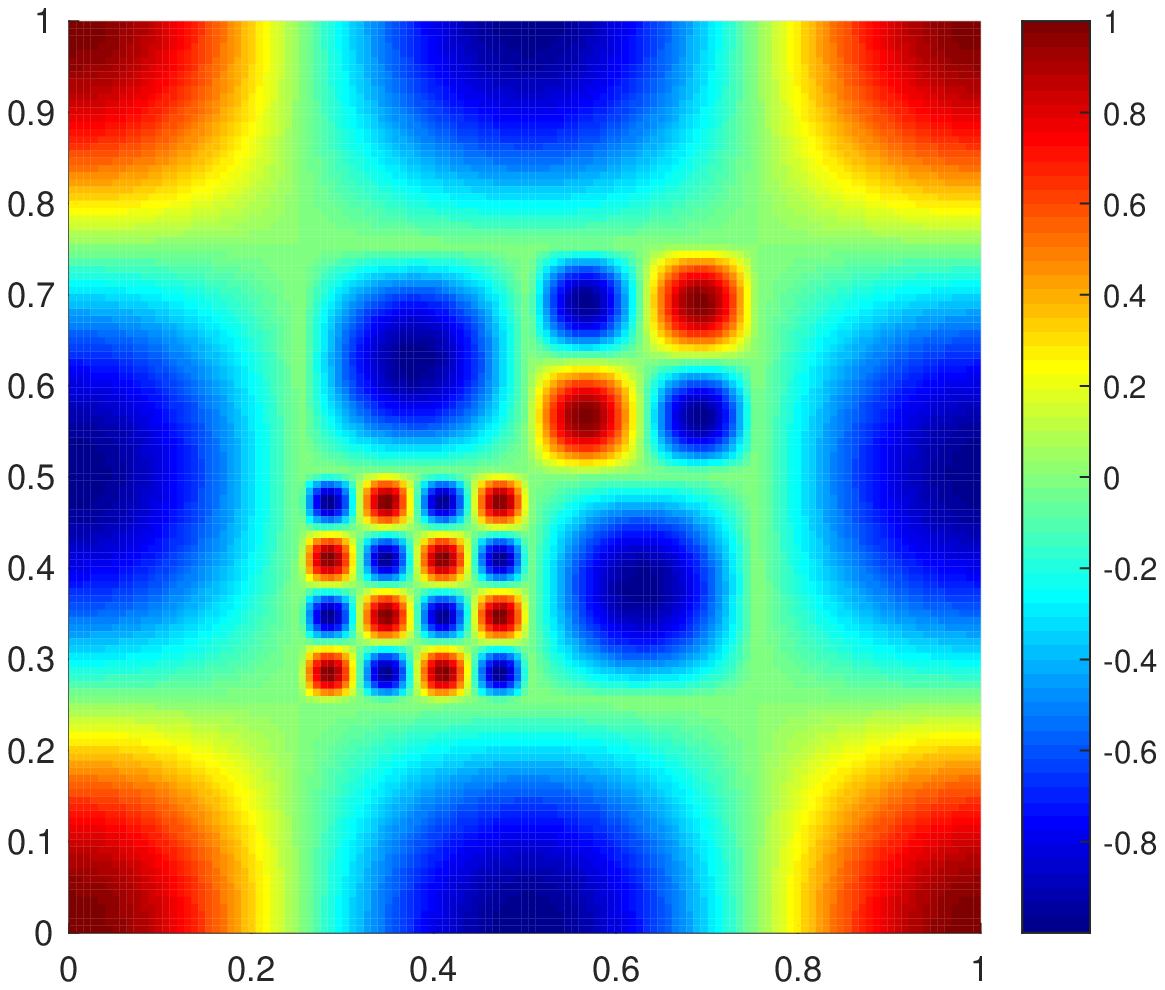}
\includegraphics[width=0.45\linewidth,trim=1cm 0cm 1.2cm 0.5cm,clip]{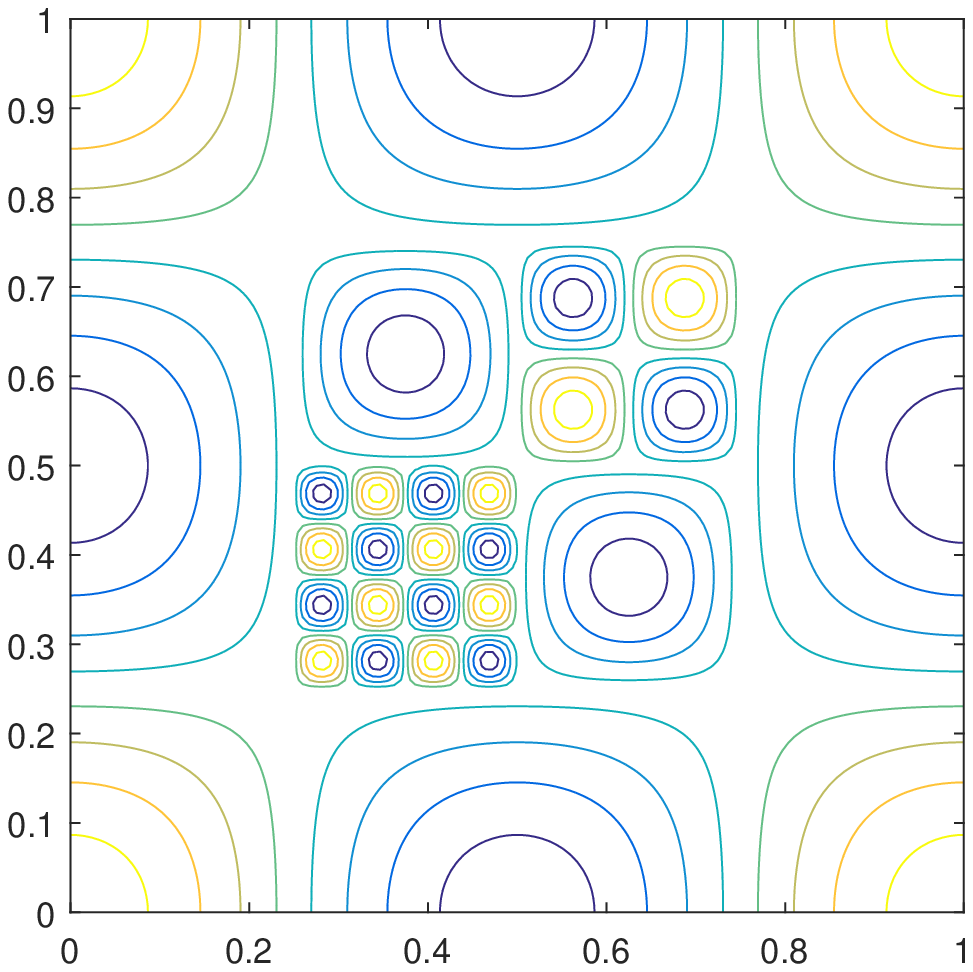}
\end{center}
\caption{The true field used to generate the data for the two-dimensional inverse problem.}
\label{f:truth_2d}
\end{figure}

\begin{figure}
\begin{center}
\includegraphics[width=0.80\linewidth,trim=2cm 1cm 2cm 1cm,clip]{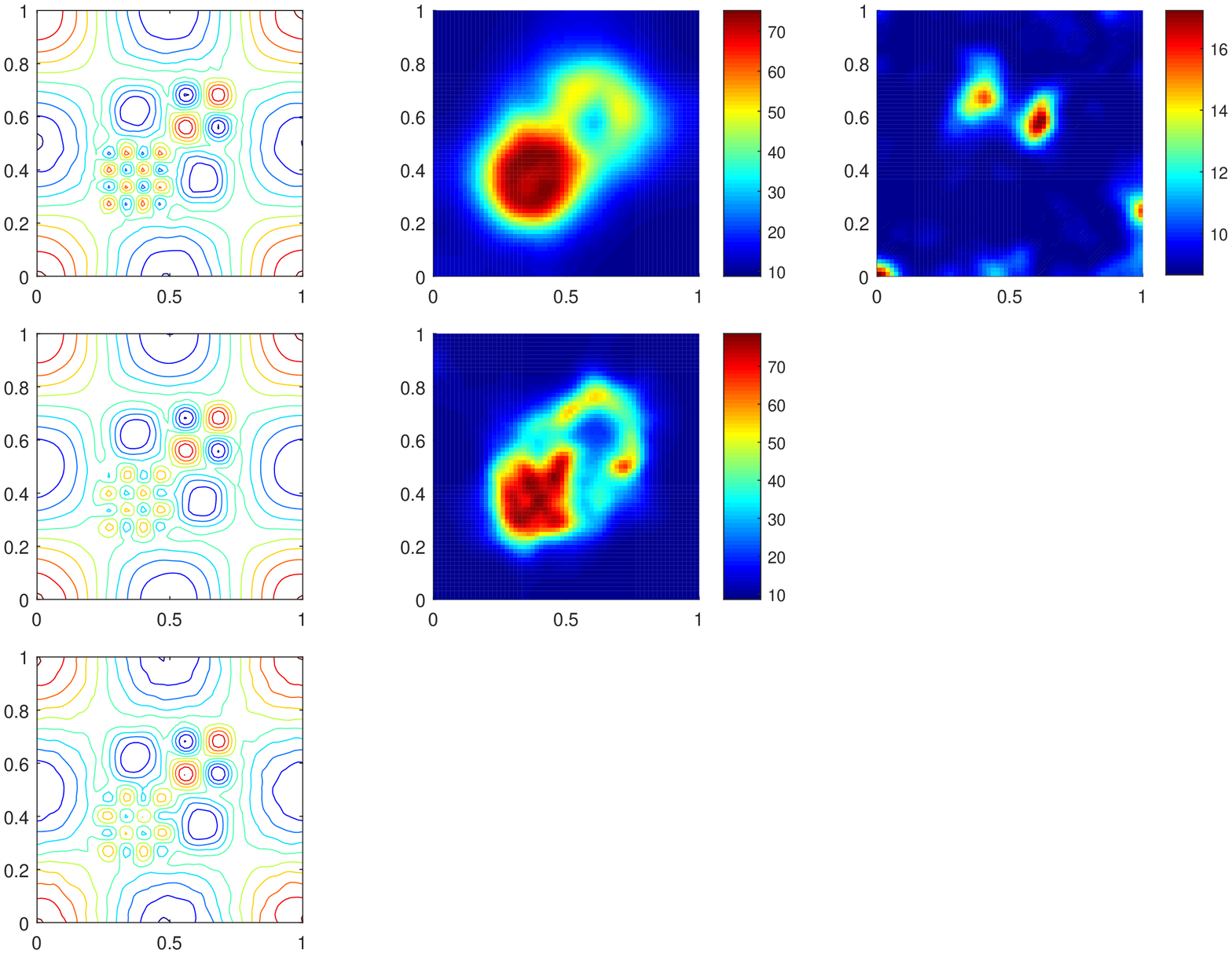}
\includegraphics[width=0.80\linewidth,trim=2cm 1cm 2cm 1cm,clip]{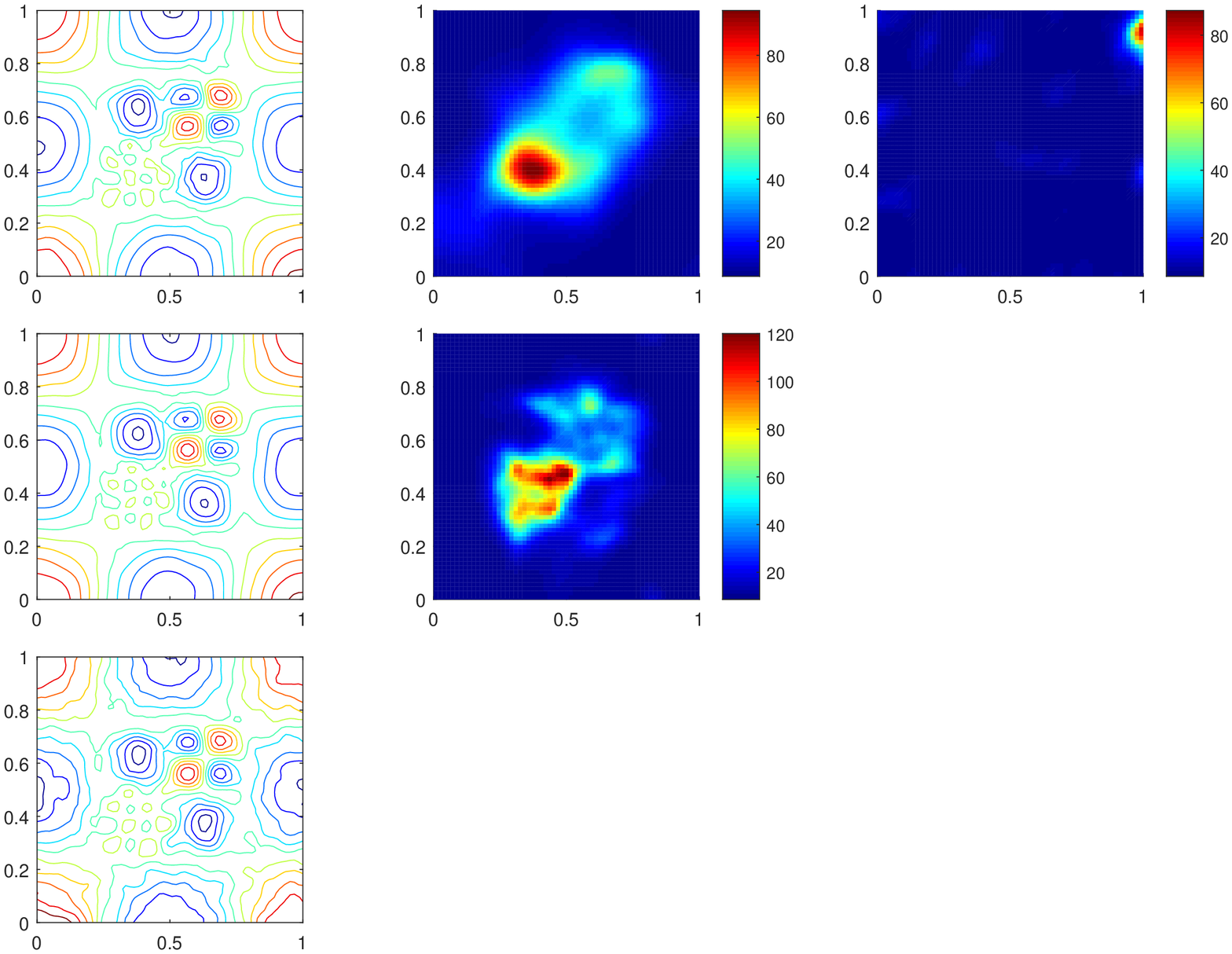}
\end{center}
\caption{Estimates of posterior means arising from two-dimensional inverse problem. (Top block) $J = 2^{10}$, (Bottom block) $J=2^8$. From left-to right,  $\mathbb{E}\big(u_N\big)$, $\mathbb{E}\big(F(u_{N-1})^{\frac{1}{2}}\big)$, $\mathbb{E}\big(F(u_{N-2})^{\frac{1}{2}}\big)$. From top-to-bottom within each block, $N=3,2,1$.}
\label{f:means_2d}
\end{figure}

\begin{table}
\caption{The $L^2$-errors $\|u^\dagger - \mathbb{E}(u_N)\|_{L^2}$ between the true field and sample means for the two-dimensional simulations shown in Figure \ref{f:means_2d}{, for different numbers of data points $J$ and layers $N$.}}
\label{tab:errs_2d}
\bgroup
\def\arraystretch{1.5}
\begin{tabular}{c|r r r}
\hline
$J$ & 1 layer & 2 layers & 3 layers\\ \hline
$2^{10}$ & 0.0856 & 0.0813 & {\bf 0.0681}\\
$2^8$ & 0.1310 & {\bf 0.1260} & 0.1279\\
\hline
\end{tabular}
\egroup
\end{table}

\section{Conclusions, Discussion, and Actionable Advice}
\label{sec:C}

%% Below is covered in Remark 3.9
%Within the context of the covariance function construction we have worked entirely
%within  the context of correlation functions, and so the trace of the covariance
%was fixed at $N$ at each step of the algorithm. It is possible, and useful in practice,
%to work with rescalings of the correlations functions, by $u_n(\cdot)$ dependent
%scales, and obtain methods outside the class examined in this paper.

%% Elliptic form not considered in final version of that paper
%One step of this process is considered in \citet{lassi}. 

\subsection{Comparison of Deep GP Constructions}
{We have considered four different constructions of deep GPs and we now discuss
their relative merits. We also consider the context of variational
inference which is popular in machine learning primarily because of its tractability.
We emphasize however that it forms an uncontrolled approximation of the true posterior
distribution and may fail to adequately represent the posterior distribution,
and uncertainty in particular.} 

{The {\bf composition} construction is the classical construction introduced in \citet{damianou2013deep}, building a hierarchy of layers using a stationary covariance function and composition. It has received the most study, and methods for variational inference have already been established. It has the advantage of scaling well with respect to data dimension $d$, however accurate sampling methods such as MCMC are intractable for large numbers of data points, due to the requirement to construct and factor dense covariance matrices at every step.

The {\bf covariance function} construction builds the hierarchy using a stationary covariance function, and iteratively modifying its associated length scale. It has the advantage that each layer can be readily interpreted as the anisotropic length-scale field of the following layer. Its scaling properties are similar to those of the composition construction, however variational inference methods for this construction have not yet been studied.

The {\bf covariance operator} construction builds the hierarchy using an SPDE representation of stationary Matern fields, and again iteratively modifies their associated length scale.  It allows for fast sampling in low data dimension $d$ via the use of PDE solvers, even when the number of data points is large. Accurate sampling via MCMC methods is tractable with this construction, due to the low cost of constructing and storing the inverse covariance (precision) matrix. Inference when $d$ is large appears to be intractable at present, due to the requirement of dense meshes for PDE solvers.

Finally, the {\bf convolution} construction builds the hierarchy via iterative convolution of Gaussian random fields. It has the advantage of being amenable to analysis, however the results of this analysis indicate that it would likely be a poor construction to use for inference due to trivial behaviour for large depth.

To summarize the numerical results on illustrative regression problems from the previous section, if the data is high quality, a small number of layers in the DGP will be sufficient as the problem becomes closer to interpolation. Conversely, if the data is low quality the likelihood is not strong enough to inform deeper layers in the DGP, and so a small number of layers is again sufficient. As a consequence,
when the data lies between these two cases, and the truth has sufficiently rich structure, the use of deeper processes may be advantageous, but
care is required to limit the number of layers employed.

}

\subsection{Summary and Future Work}
There are a number of interesting ways in which this work may be generalized.
 Within the context
of covariance operators it is of interest to construct covariances $C(u)$ which
are defined as $L^{-\alpha}$ with $L$ being the divergence form elliptic operator
$$Lu=-\nabla \cdot\bigl(F(u)\nabla u\bigr).$$ 
Such a construction allows for the conditional distributions of the layers to be viewed 
as stationary on deformed spaces \citet[\S 3.4]{matern_spde}, or to incorporate anisotropy 
in specific directions \citet[\S 3.1]{roininen2014whittle}. Similar notions of anisotropy in different directions can be incorporated into the covariance function formulation by choosing the length scale $\Sigma(z)$ different to a multiple of the identity matrix.
{Additionally, we could consider 
a non-zero mean in the iteration (\ref{eq:cg2}), as in  \citet{duvenaud2014avoiding,salimbeni2017doubly}, allowing for forcing of the system. For example, with the choice $m(u_n) = u_n$ and a rescaling of the covariance, we obtain the ResNet-type iteration
\[
u_{n+1} = u_n + \sqrt{\Delta t} L(u_n)\xi_{n+1}.
\]
This may be viewed as a discretization of the continuous-time stochastic
differential equation 
\[
{\dee u}= L(u)L(u)^\top{\dee W},
\]
analogously to what has been considered for neural networks \citet{haber2017stable}. Study of these systems could be insightful, for example deriving conditions to ensure a lack of ergodicity and hence arbitrary depth.}
{As before $\top$ denotes the adjoint operation.}

And finally it is possible to consider processes outside the four categories considered
here; for example the one-step transition from $u_n$ to $u_{n+1}$ might be defined
via stochastic integration against i.i.d. Brownian motions.

We have shown how a number of ideas in the literature may be recursed to produce
deep Gaussian processes, different from those appearing in \citet{damianou2013deep}.
We have studied the effective depth of these processes, either through demonstrating
ergodicity, or through showing convergence to a trivial solution (such as $0$ or
$\infty$). Together these results demonstrate that, as also shown in
\citet{duvenaud2014avoiding} for the original construction of deep Gaussian processes,
care is needed in order to design processes with significant depth. Nonetheless,
even a few layers can be useful for inference purposes, and we have demonstrated
this also. d{It is an interesting question to ask precisely how the approximation power and effective depth are affected by the number of layers of the process, both in the non-ergodic case, and in the ergodic case before stationarity has been reached.} 

{We also emphasize that the analysis in the paper is based solely on the deep Gaussian process $u_n$, and not the conditioned process $u_n|y$ in the inference problem with observed data $y$. The ergodicity properties of $u_n$ do not directly carry over to $u_n|y$. As we have seen in the numerical experiments, the number of layers required in the inference problem in practice depends on the information content in the observed data $y$, and the analysis in this paper does not fully answer the question as to how many. The results in this paper do show, however, that in the case of ergodic constructions, the expressive power of the {\em prior} distribution in the inference problem does not increase past a certain number of layers. This provides some justification for using only a moderate number of layers in a deep Gaussian process prior in inference problems.}

There are interesting approximation theory questions around deep processes,
such as those identified in the context of neural networks in \citet{pinkus1999approximation}.
There are also interesting questions around the use of these deep processes
for inversion; in particular it seems hard to get significant value from using
depth of more than two or three layers for noisy inverse problems. 
On the algorithmic side the issue of efficiently sampling these deep processes (even
over only two layers), when conditioned on possibly nonlinear observations
remains open. We have used non-centred parameterizations because these
{may} be sampled using function-space MCMC \citet{cotter2013mcmc,omiros}; 
but centred methods, or mixtures, may be desirable for some applications.

\bibliographystyle{natbib}
\bibliography{bib_dgp}

\section*{Appendix}

\begin{proof}[Proof of Proposition \ref{thm:nonstat}] The stationary kernel $\rho_\mathrm{S}$ is positive definite by Assumption \ref{ass:posdef_stat}, and so by \citet[Theorem 7.14]{wendland}, we have
\[
\rho_\mathrm{S}(r) = \int_{0}^\infty \exp(-r^2 t)\,\mathrm{d} \nu(t) \quad \text{for all } r \in [0,\infty),
\]
for a finite, non-negative Borel measure $\nu$ on $[0,\infty)$ that is not concentrated at 0 (i.e. it is not a multiple of the Dirac measure centred at 0). 

For any $x \in \R^d$ and $t \in [0,\infty)$, let us now define the matrix $ \tilde \Sigma_t(x) := (4t)^{-1} \Sigma(x)$ and the functions
\[
K_{x,t}(z) = \frac{1}{(2 \pi)^{d/2} |\tilde \Sigma_t(x)|^{1/2}}\exp\bigg(-\frac{1}{2}(x- z)^T \tilde \Sigma_t(x)^{-1} (x - z)\bigg).
\]
Here $|\cdot|$ denotes determinant and so the preceding is simply
an expression for a normal density with mean $x$ and covariance matrix $\tilde \Sigma_t(x)$ when $t>0$; at $t=0$, we simply have $K_{x,t}(z) = 0$, for all $x,z \in \R^d$. Then $\rho(x,x')$ is given by
\begin{align*}
&\frac{2^\frac{d}{2} |\Sigma(x)|^\frac{1}{4} |\Sigma(x')|^\frac{1}{4}}{|\Sigma(x) + \Sigma(x')|^\frac{1}{2}} \rho_\mathrm{S}\Big(\sqrt{Q(x,x')}\Big) 
= \frac{2^\frac{d}{2} |\Sigma(x)|^\frac{1}{4} |\Sigma(x')|^\frac{1}{4}}{|\Sigma(x) + \Sigma(x')|^\frac{1}{2}} \int_{0}^\infty \exp\big(- t Q(x,x')\big)\,\mathrm{d} \nu(t) \\
&= \frac{2^\frac{d}{2} |\Sigma(x)|^\frac{1}{4} |\Sigma(x')|^\frac{1}{4}}{|\Sigma(x) + \Sigma(x')|^\frac{1}{2}}  \int_{0}^\infty \exp\bigg(- t (x-x')^T\left(\frac{\Sigma(x) + \Sigma(x')}{2}\right)^{-1} (x-x')\bigg)\,\mathrm{d} \nu(t) \\
&= 2^\frac{d}{2} \int_{0}^\infty \frac{ |\tilde \Sigma_t(x)|^\frac{1}{4} |\tilde \Sigma_t(x')|^\frac{1}{4}}{|\tilde \Sigma_t(x) + \tilde \Sigma_t(x')|^\frac{1}{2}}   \exp\bigg(- \frac{1}{2} (x-x')^T\left(\tilde \Sigma_t(x) + \tilde \Sigma_t(x')\right)^{-1} (x-x')\bigg)\, \mathrm{d} \nu(t) \\
&= (2 \pi)^\frac{d}{2} 2^\frac{d}{2}  \int_{0}^\infty |\tilde \Sigma_t(x)|^\frac{1}{4} |\tilde \Sigma_t(x')|^\frac{1}{4} \int_{\R^d} K_{x,t}(z) K_{x',t}(z)\,\mathrm{d}z\, \mathrm{d} \nu(t),
\end{align*}
where in the last step, we have used the fact that the convolution $\int_{\R^d} K_{x,t}(z) K_{x',t}(z)\,\mathrm{d} z$ can be calculated explicitly using properties of normal random variables. More precisely, we have
\[
\int_{\R^d} K_{x,t}(z) K_{x',t}(z)\,\mathrm{d} z = \int_{\R^d} p_X(z-x) p_{X'}(z)\, \mathrm{d} z = \int_{\R^d} p_{X,X'} (z-x, z) \,\mathrm{d} z,
\]
where $p_X$ is the density of $X \sim N(0, \tilde \Sigma_t(x))$, $p_{X'}$ is the density of $X' \sim N(x', \tilde \Sigma_t(x'))$ and $X$ and $X'$ are independent. The change of variable from $X,X'$ to $W,X'$, where $W = X'-X$, has Jacobian 1, and so
\[
\int_{\R^d} p_{X,X'} (z-x, z)\,\mathrm{d} z = \int_{\R^d} p_{W,X'}(z-(z-x), z)\,\mathrm{d} z = \int_{\R^d} p_{W,X'} (x, z)\,\mathrm{d} z = p_W(x).
\]
Since $W = X' - X \sim N(x', \tilde \Sigma_t(x) + \tilde \Sigma_t(x'))$, we hence have
\begin{align*}
\int_{\R^d} &K_{x,t}(z) K_{x',t}(z)\,\mathrm{d} z \\
&= \frac{1}{(2 \pi)^\frac{d}{2} |\tilde \Sigma_t(x) + \tilde \Sigma_t(x')|^\frac{1}{2}}  \exp\bigg(- \frac{1}{2} (x-x')^T\left(\tilde \Sigma_t(x) + \tilde \Sigma_t(x')\right)^{-1} (x-x')\bigg),
\end{align*}
as required.

Now, for any $b \in \R^N$ and pairwise distinct $\{x_i\}_{i=1}^N$, we then have
\begin{align*}
&\sum_{i=1}^N \sum_{j=1}^N b_i b_j \rho(x_i,x_j) \\
&= (2 \pi)^\frac{d}{2} 2^\frac{d}{2} \sum_{i=1}^N \sum_{j=1}^N b_i b_j \int_{0}^\infty \int_{\R^d}  |\tilde \Sigma_t(x_i)|^\frac{1}{4} K_{x_i,t}(z)  |\tilde \Sigma_t(x_j)|^\frac{1}{4} K_{x_j,t}(z)\,\mathrm{d} z\, \mathrm{d} \nu(t) \\
&= (2 \pi)^\frac{d}{2} 2^\frac{d}{2} \int_{0}^\infty \int_{\R^d}  \left(\sum_{i=1}^N b_i |\tilde \Sigma_t(x_i)|^\frac{1}{4} K_{x_i,t}(z) \right)^2 \mathrm{d} z\, \mathrm{d} \nu(t) \\
& \geq 0,
\end{align*}
since the Borel measure $\nu$ is finite and non-negative. It remains to show that strict inequality also holds.

Firstly, we note that $|\tilde \Sigma_0(x_i)|^\frac{1}{4} K_{x_i,0}(z) = 0$, for all $x_i,z \in \R^d$, which means that the integrand with respect to $t$
is identically equal to zero at $t=0$. 
Secondly, we note that the points $\{x_i\}_{i=1}^N$ are pairwise distinct 
and the functions 
$\{|\tilde \Sigma_t(x_i)|^\frac{1}{4} K_{x_i,t}(\cdot)\}_{i=1}^N$ 
are hence linearly independent for any $t \in (0,\infty)$. It is 
thus impossible to make the integrand with respect to $z$
identically equal to $0$ for a.e. $z \in \R^d$. As a consequence the integrand
with respect to $t$ is positive for all $t \in (0,\infty)$.
Since we know that the measure $\nu$ is not concentrated at $0$
this completes the proof that $\rho$ is positive definite on 
$\R^d \times \R^d$, for any $d \in \N$.

Finally, we note that the kernel $\rho$ is clearly non-stationary, and is a correlation function since $\rho(x,x)=1$, for any $x \in \R^d$.
\end{proof}

\begin{proof}[Proof of Proposition \ref{lem:posdef}]
We note that the definition of positive definite in Assumptions \ref{ass:posdef_stat}(i) refers
only to behaviour of the kernel on a finite set of pairwise distinct points
$\{x_i\}_{i=1}^N.$
By Assumption \ref{ass:posdef}(i), the function $G$ is non-negative and bounded. 
%which means that $G(u_{n}(z)) \in [0,\infty)$, for all bounded functions $u_{n}$ and arbitra%ry $z \in \R^d$. 
If $G(z) > 0$ for all $z \in \R^d$, then the matrix $\Sigma(z)$ is positive definite for all $z \in \R^d$, and the fact that $\rho(\cdot,\cdot)$ is positive definite follows directly from Proposition \ref{thm:nonstat}. 

It remains to investigate the case where $G(z) = 0$ for some $z \in \R^d$. We will prove 
that $\rho(\cdot,\cdot)$ is positive definite by showing that the correlation matrix 
$\bR$, with entries $\bR_{ij} = \rho( x_i, x_j)$, is positive definite for any pairwise disjoint points $\{x_i\}_{i=1}^N$. 
%Without loss of generality, we will study the case $F(u_{n}(x_1)) = 0$ and $F(u_{n}(x_i)) > 0$, for $i=2,\dots, N$. The proof easily adapts to the case where $F(u_{n}(x_i)) = 0$, for $i \neq 1$, or $F(u_{n}(x_i)) = 0$ for several indices $i$.
{Without loss of generality, we will study the case $G(x_1) = 0$; the proof easily adapts to the case where $G(x_i) = 0$, for $i \neq 1$. To define $\rho( x_1, x_j)$ in this case, we start by assuming $G(x_1) > 0, G(x_j) > 0$, and then take limits.

With $\Sigma(z) = G(z) \mathrm{I}_d$, we have 
\begin{align*}
Q(x_1,x_j) &= (x_1-x_j)^T\left(\frac{\Sigma(x_1) + \Sigma(x_j)}{2}\right)^{-1} (x_1-x_j) \\
&= 2 \|x_1-x_j\|_2^2 \Big(G(x_1) + G(x_j)\Big)^{-1},
\end{align*}
where $\|\cdot\|_2$ is the Euclidean norm, and
\[
\frac{2^\frac{d}{2} \det(\Sigma(x_1))^\frac{1}{4} \det(\Sigma(x_j))^\frac{1}{4}}{\det(\Sigma(x_1) + \Sigma(x_j))^\frac{1}{2}} = \left(\frac{4 G(x_1) G(x_j)}{\big( G(x_1) + G((x_j) \big)^2} \right)^{\frac{d}{4}}.
\] 
We now study separately three cases:
\begin{itemize}
\item[i)] $x_j = x_1$:  we have
\begin{equation}
\label{eq:az1}
\lim_{G(x_1)\rightarrow 0} \left(\frac{4 G(x_1) G(x_1)}{\big( G(x_1) + G(x_1) \big)^2} \right)^{\frac{d}{4}} = \lim_{G(x_1)\rightarrow 0} 1 =1,
\end{equation} 
and so using the algebra of limits, the continuity of 
$\rho_\mathrm{S}$, \eqref{eq:az1} and the fact that $\rho_\mathrm{S}(0)=1$, we have
\[
\lim_{G(x_1)\rightarrow 0} \rho(x_1,x_1) = \lim_{G(x_1) \rightarrow 0} \rho_\mathrm{S}\Big(\sqrt{Q(x_1,x_1)}\Big) = \rho_\mathrm{S}(0) = 1.
\]
\item[ii)] $x_j \neq x_1$ and $G(x_j)>0$: we have
\[
\lim_{G(x_1) \rightarrow 0} Q(x_1,x_j) = 2 \|x_1-x_j\|_2^2 \Big(G(x_j))\Big)^{-1},
\]
and 
\begin{equation}
\label{eq:az2}
\lim_{G(x_1) \rightarrow 0} \left(\frac{4 G(x_1) G(x_j)}{\big( G(x_1) + G(x_j) \big)^2} \right)^{\frac{d}{4}}
= 0.
\end{equation}
%\left(\frac{4 G(x_1) G(x_j)}{\big( G(x_1) + G(x_j) \big)^2} \right)^{\frac{d}{4}}
Thus, using the continuity of $\rho_\mathrm{S}$, together with \eqref{eq:az2} and the algebra of limits, we have
%with $\lim_{r \rightarrow \infty} \rho_\mathrm{S}(r)=0$, we have for $x_j \neq x_1$
\[
\lim_{G(x_1)\rightarrow 0} \rho(x_1,x_j)  = 0.
%=\lim_{F(u_{n}(x_1)) \rightarrow 0} \left(\frac{4 F(u_{n}(x_1)) F(u_{n}(x_j))}{\big( F(u_{n}(x_1)) + F(u_{n}(x_j)) \big)^2} \right)^{\frac{d}{4}}
%\rho_\mathrm{S}\Big(\sqrt{Q(x_1,x_j)}\Big)
%\lim_{r \rightarrow \infty} \rho_\mathrm{S}(r) = 0.
\]
\item[iii)] $x_j \neq x_1, G(x_j) = 0$: %in this case, we start by assuming$G(x_1) > 0, G(x_j) > 0$ and take limits. W
we obtain 
\[
\lim_{G(x_1), G(x_j) \rightarrow 0} Q(x_1,x_j) = \infty, 
\]
which by Assumptions \ref{ass:posdef}(ii) implies that
\[
\lim_{G(x_1), G(x_j) \rightarrow 0} \rho_\mathrm{S}\Big(\sqrt{Q(x_1,x_j)} \Big) = 0.
\]
Since $(a+b)^2 \geq 4ab$ for any positive numbers $a$ and $b$, we have
\[
0 \leq \left(\frac{4 G(x_1) G(x_j)}{\big( G(x_1) + G((x_j) \big)^2} \right)^{\frac{d}{4}} \leq 1,
\]
for any $G(x_1) > 0, G(x_j) > 0$, and hence
\[
\lim_{G(x_1), G(x_j) \rightarrow 0} \rho(x_1,x_j)  = 0.
%=\lim_{F(u_{n}(x_1)) \rightarrow 0} \left(\frac{4 F(u_{n}(x_1)) F(u_{n}(x_j))}{\big( F(u_{n}(x_1)) + F(u_{n}(x_j)) \big)^2} \right)^{\frac{d}{4}}
%\rho_\mathrm{S}\Big(\sqrt{Q(x_1,x_j)}\Big)
%\lim_{r \rightarrow \infty} \rho_\mathrm{S}(r) = 0.
\]
\end{itemize}
%Thus, 
%Similarly, using the continuity and non-negativeness of $\rho_\mathrm{S}$, together with \eqref{eq:az2},
%%with $\lim_{r \rightarrow \infty} \rho_\mathrm{S}(r)=0$, we have for $x_j \neq x_1$
%\[
%\lim_{F(u_{n}(x_1))\rightarrow 0} \rho(x_1,x_j;u_{n}) =\lim_{F(u_{n}(x_1)) \rightarrow 0} \left(\frac{4 F(u_{n}(x_1)) F(u_{n}(x_j))}{\big( F(u_{n}(x_1)) + F(u_{n}(x_j)) \big)^2} \right)^{\frac{d}{4}}
%\rho_\mathrm{S}\Big(\sqrt{Q(x_1,x_j)}\Big) = 0.
%%\lim_{r \rightarrow \infty} \rho_\mathrm{S}(r) = 0.
%\]
%Since $\rho(x_1,x_j;u_{n}) \geq 0$ by Assumptions \ref{ass:posdef}, we can conclude that $\lim_{F(u_{n}(x_1))\rightarrow 0} \rho(x_1,x_j;u_{n}) =0$.

Hence, when $G(x_i) > 0$, for $i=2,\dots, N$, we have $\lim_{G(x_1)\rightarrow 0} \bR = \bR^*$, where the matrix $\bR^*$ has the first row and column equal to the first basis vector $e_1 = (1,0,0,\dots,0) \in \R^N$, and the remaining submatrix $\bR^*_{N-1} \in \R^{N-1 \times N-1}$ with entries $\rho(x_i, x_j)$, for $i,j=2,\dots,N$. The matrix $\bR^*_{N-1}$ is positive definite by Proposition \ref{thm:nonstat}, from which we can conclude that $\bR^*$ is positive definite also. A similar argument holds when $G(x_i) = 0$ for one or more indices $i \in \{2, \dots, N\}$.}
\end{proof}

\end{document}